\documentclass[11pt,a4paper,leqno,fullpage]{article}
\usepackage[usenames,dvipsnames]{color}
\usepackage{amsmath, amssymb, amscd, amsthm, amsfonts, amsbsy, appendix}
\usepackage{bbm, stmaryrd, mathrsfs, bm, linearb, upgreek, skak, textcomp, mathdots}
\usepackage{mathtools, titlesec, extarrows, tocloft, url, dsfont, verbatim, latexsym}
\usepackage[intoc]{nomencl}
\usepackage[nottoc]{tocbibind}
\usepackage[mathscr]{euscript}
\usepackage[all, cmtip]{xy} 
\usepackage{graphicx, pgf}
\usepackage[T1]{fontenc}

\DeclareMathAlphabet{\mathpzc}{OT1}{pzc}{m}{it} 

\DeclareMathAlphabet{\mathpzc}{OT1}{pzc}{m}{it}

\title{{\LARGE Homotopy groups of highly connected manifolds}}
\author{Samik Basu, Somnath Basu}
\date{}
\usepackage{pdfpages}
\usepackage[pdfauthor={Samik Basu, Somnath Basu},%
        pdftitle={Homotopy groups of highly connected manifolds},%
        pdftex]{hyperref}
\usepackage{hyperref}
\hypersetup{
	colorlinks,
	citecolor=red,
	filecolor=black,
	linkcolor=blue,
	urlcolor=black
}

\theoremstyle{plain}
\newtheorem{theorem}{Theorem}[section]
\newtheorem{prop}[theorem]{Proposition}
\newtheorem{lemma}[theorem]{Lemma}
\newtheorem{cor}[theorem]{Corollary}

\theoremstyle{definition}
\newtheorem{defn}[theorem]{Definition}
\newtheorem{rmk}[theorem]{Remark}
\newtheorem{exam}[theorem]{Example}

\newcommand{\C}{{\mathbb C}}

\newcommand{\Hy}{{\mathbb H}}

\newcommand{\Z}{{\mathbb{Z}}}

\newcommand{\Q}{{\mathbb Q}}

\newcommand{\Oc}{{\mathbb O}}

\newcommand\LL{{\mathcal L}}
\newcommand\MM{{\mathcal M}}

\newcommand\PP{{\mathcal P}}

\newcommand\RR{{\mathcal R}}

\newcommand\PMF{{\PP\kern-2pt\MM\FF}}

\newcommand\PML{{\PP\kern-2pt\MM\LL}}

\newcommand\ep{\epsilon}

\newcommand{\fsubd}{\mathrel{{\scriptstyle\searrow}\kern-1ex^d\kern0.5ex}}
\newcommand{\bsubd}{\mathrel{{\scriptstyle\swarrow}\kern-1.6ex^d\kern0.8ex}}
\newcommand{\fsubeq}{\mathrel{\raise-.7ex\hbox{$\overset{\searrow}{=}$}}}
\newcommand{\bsubeq}{\mathrel{\raise-.7ex\hbox{$\overset{\swarrow}{=}$}}}

\newcommand{\tsh}[1]{\left\{\kern-.9ex\left\{#1\right\}\kern-.9ex\right\}}

\newcommand{\dimn}{\mathit{dim}}

\newcommand{\Lie}{\mathit{Lie}}

\newcommand{\rank}{\mathit{rank}}

\newcommand{\bgc}{\begin{center}}
\newcommand{\edc}{\end{center}}

\newtheorem*{thma}{Theorem A}
\newtheorem*{thmb}{Theorem B}
\newtheorem*{thmc}{Theorem C}
\newtheorem*{thmd}{Theorem D}

\numberwithin{equation}{section}

\makeatletter
\newcommand{\subjclass}[4]{%
  \let\@oldtitle\@title%
  \gdef\@title{\@oldtitle\footnotetext{2010 \emph{Mathematics subject classification.} #1, #2 (primary); #3, #4 (secondary).}}%
}
\newcommand{\keywords}[5]{%
  \let\@@oldtitle\@title%
  \gdef\@title{\@@oldtitle\footnotetext{\emph{Key words and phrases.} #1, #2, #3, #4, #5.}}%
}
\makeatother

\begin{document}

\subjclass{16S37}{55Q52}{55P35}{57N15}
\keywords{Homotopy groups}{Koszul duality}{Loop space}{Moore conjecture}{Quadratic algebra}

\maketitle

\begin{abstract}
In this paper we give a formula for the homotopy groups of $(n-1)$-connected $2n$-manifolds as a direct sum of homotopy groups of spheres in the case the $n^\textup{th}$ Betti number is larger than $1$. We demonstrate that when the $n^\textup{th}$ Betti number is $1$ the homotopy groups might not have such a decomposition. The techniques used in this computation also yield formulae for homotopy groups of connected sums of sphere products and CW complexes of a similar type. In all the families of spaces considered here, we establish a conjecture of J. C. Moore.
\end{abstract}

\vspace*{0.2cm}

\tableofcontents
\vspace*{0.75cm}

\setcounter{section}{0}
\setcounter{tocdepth}{2}

\section{Introduction}

One of the primary invariants associated to a topological space are its homotopy groups. Although easier to define than homology groups, these are notoriously hard to compute even for simple examples. For a simply connected space $X$, all the homotopy groups are abelian. In his celebrated thesis \cite{Ser51}, Serre proved that $\pi_i(S^n)$ is torsion unless $i=0,n$, or when $n$ is even $\pi_{2n-1}(S^n)$ has a $\Z$-summand. For $n=2$ this goes back to the existence of the Hopf map $f:S^3\to S^2$. In fact, Serre also showed that a simply connected, finite CW complex $X$ has infinitely many non-zero homotopy groups. He conjectured that such a space with non-trivial $\Z/p\Z$ cohomology has the property that $\pi_n(X)$ contains $\Z/p\Z$ for infinitely many values of $n$. This was proved later by McGibbon and Neisendorfer \cite{McGN85} as an application of an amazing theorem of Miller \cite{Mill84}, which itself was proving a special case of the Sullivan's conjecture on the fixed point set in group actions of a finite group.

Following Serre, computing homotopy groups of spheres has been a point of widespread research. Results about the homotopy groups of manifolds and associated CW complexes have also been of interest. In such cases results are typically of the form that the homotopy groups of a manifold are related to the homotopy groups of spheres via some formula. Leaving aside aspherical manifolds, the real and complex projective spaces have homotopy groups which are well-understood as they support a sphere bundle whose total space is a sphere. In 1955, Hilton \cite{Hil55} computed the homotopy of groups of a finite wedge of spheres\footnote{Hilton calls this a finite {\it union} of spheres.}, as a direct sum of homotopy groups of spheres.  The various spheres were mapped to the wedge by iterated Whitehead products inducing a decomposition of the (based) loop space. In 1972, Milnor's unpublished note \cite{Ada72} generalized Hilton's result to prove that the loop space of the suspension of a wedge of spaces can be written as an infinite product of loop spaces of suspensions of smash products. This version is usually referred to as the Hilton-Milnor Theorem. 

This paper deals with explicit computations of homotopy groups of certain manifolds and allied applications. Our main results involve computations for $(n-1)$-connected $2n$-manifolds and similar CW complexes which are next in line (in terms of complexity of cell structures) from a wedge of spheres. The $(n-1)$-connected $2n$-manifolds were considered by Wall \cite{Wall64} as generalizations of simply connected $4$-manifolds. The theory of these four manifolds has in itself been a topic of considerable interest. The classification of simply connected four manifolds upto homotopy type was achieved in the early works of Whitehead and Milnor and the homeomorphism classification is a celebrated result of Freedman. 

A simply connected $4$-manifold $M$ may be expressed upto homotopy as a CW complex with $r$ two-cells and one four-cell, where $r$ is the second Betti number of $M$. Therefore the homotopy groups of $M$ must be determined by the integer $r$ and the attaching map of the four-cell. Curiously, if $r\geq 1$, one may construct a circle bundle over $M$ whose total space is a connected sum of $(r-1)$ copies of $S^2\times S^3$ (cf. \cite{BaBa15,DuLi05}). It follows that the homotopy groups, upto isomorphism, depend only on the integer $r$ and {\it not} on the attaching map.  

The above arguments present us with two issues. One, whether it is possible to write down an expression for the homotopy groups in terms of homotopy groups of spheres by computing the homotopy groups of the relatively explicit CW complex $\#^{r-1}(S^2 \times S^3)$. Two, the arguments used above are typically geometric relying on the classification of spin $5$-manifolds by Smale \cite{Sma62}. Thus, a natural question is whether one can prove the same theorem using homotopy theoretic methods. In this paper we answer both these questions. 

For the second question above, we compute the homotopy groups of a $(n-1)$-connected $2n$-manifold provided the Betti number of the middle dimension is at least $2$. We prove the following result (cf. Theorem \ref{htpy} and Theorem \ref{htpyform}).
\begin{thma}
{\it Let $M$ be a closed $(n-1)$-connected $2n$-manifold with $n\geq 2$ and $n^\textit{th}$ Betti number $r\geq 2$.  \\
\textup{(a)} The homotopy groups of $M$ are determined by $r$.\\
\textup{(b)} The homotopy groups of $M$ can be expressed as a direct sum of homotopy groups of spheres.   
}
\end{thma}
\noindent In fact, the number of $\pi_jS^l$ occuring in $\pi_jM$ is also determined (cf. Theorem \ref{htpyform}). From this expression the results in \cite{BaBa15} about the rational homotopy groups may also be deduced. A result of such a computational nature is of general significance and interest. Quite curiously and a bit unexpectedly, the case when the $n^\textup{th}$ Betti number is $1$ turns out to be different. Note that this implies that there exists a Hopf invariant one element in $\pi_{2n-1}S^n$ so that $n=2,4,$ or $8$. When $n=2$, Theorem A holds for the case $r=1$. However for $n=4$ or $8$, Theorem A does not extend to the case $r=1$. In fact, a hidden relationship with the homotopy associativity of $H$-space structures on $S^3$ and $S^7$ is revealed. However, when some primes are inverted one obtains an analogous version of Theorem A for the case $r=1$. 

The primary techniques used in the proof of Theorem A are quadratic algebras, Koszul duality of associative algebras and quadratic Lie algebras. These techniques are used to compute the homology of the loop space of the manifold using the cobar construction on the homology of $M$. A naive version of these ideas was used in an earlier paper \cite{BaBa15}. 

The techniques developed to prove Theorem A may be used to thoroughly analyze homotopy groups of other spaces. Let $T$ be the manifold obtained by connected sum of sphere products $\{S^{p_j}\times S^{n-p_j}\}_{j=1}^r$, where each of the constituent sphere is simply connected. That is, $T$ is determined by the numbers $p_j$ and a choice of orientation on each $S^{p_j}\times S^{n-p_j}$. We have the following result (cf. Theorem \ref{htpyT} and Theorem \ref{htpybetti}).
\begin{thmb}
{\it Let $T$ be a connected sum of sphere products.  \\
\textup{(a)} The homotopy groups of $T$ do not depend on the choice of orientation used in the connected sum.\\
\textup{(b)} The homotopy groups of $T$ can be expressed as a direct sum of homotopy groups of spheres.   
}
\end{thmb}
\noindent As before, the number of copies of $\pi_jS^l$ in $\pi_jT$ is computed (cf. Theorem \ref{htpyTform}). We also apply these ideas to CW complexes of the type $X= \vee_r S^n \cup e^{2n}$. We consider the quadratic form $Q$ given by the cup product of $X$ and obtain similar results after inverting a finite set  of primes (cf. Theorem \ref{cellhtpy} and Theorem \ref{cellhtpyform}). 
\begin{thmc}
{\it Let $X$ be as above and suppose that $\mathit{Rank}(Q\otimes \Q) \geq 2$. Then there exists a finite set of primes $\Pi_Q$ so that \\
\textup{(a)} The homotopy groups $\pi_j X \otimes \Z_{(p)}$ depends only on $r$ if $p\notin \Pi_Q$,\\
\textup{(b)} The homotopy groups  $\pi_j X\otimes \Z_{(p)}$ can be expressed as a direct sum of homotopy groups of spheres if $p\notin \Pi_Q$.   
}
\end{thmc}
The condition $\mathit{Rank}(Q\otimes \Q) \geq 2$ is analogous to the condition that the $n^\textup{th}$ Betti number is at least $2$ in Theorem A.

The homotopy groups of a simply connected space, being abelian, splits into the torsion part and the free part. The free part may be studied via rational homotopy theory. A dichotomy presents itself here - either the sum of the ranks of the rational homotopy groups is finite (such spaces are called {\it rationally elliptic}) or it is not. In the latter case, the spaces are called {\it rationally hyperbolic} and it is known that the partial sums of the ranks of the rational homotopy groups grow at least exponentially. It follows that a space is rationally elliptic if there exists $l$ such that the dimension of $\pi_i X\otimes \Q$ is at most $l$ for any $i$. Otherwise, a space is rationally hyperbolic and no such $l$ exists. 

To study the torsion part, an analysis of the $p$-primary torsion of $\pi_\ast X$ for each prime $p$ has to be made. We may treat the dichotomy in rational homotopy groups as a property of the $0$-primary part. It is therefore natural to ask for suitable conditions on the space such that for a fixed prime $p$, there is a {\it finite $p$-exponent}, i.e., an exponent $s$ such that $p^s$ annihilates the $p$-torsion in $\pi_\ast X$ and $s$ is the smallest such integer. A precise formulation (\cite{FHT01}, pp. 518) is as follows.\\[0.2cm]
\noindent{\bf Conjecture (Moore)}\\
{\it Let $X$ be a finite, simply connected CW complex. It is rationally elliptic if and only if for each prime $p$ some $p^r$ annihilates all the $p$-primary torsion in $\pi_\ast X$.}\\[0.2cm]
Along with Serre's conjecture on $p$-torsion in homotopy groups, this conjecture has been a focus of research in unstable homotopy theory since the late 1970's when it was introduced by J. C. Moore. It has been verified for spheres \cite{CMN79}. It is also known \cite{McGW86} that for elliptic spaces, away from a finite set of primes, Moore conjecture holds.

In all the families mentioned in the results so far, the above conjecture can be established. We have the following result (cf. Theorem \ref{Moorehcm}, Theorem \ref{MooreT}, Theorem \ref{MooreCW} and Remark \ref{MooreX}).
\begin{thmd}
{\it \textup{(a)} Consider any closed $(n-1)$-connected $2n$-manifold $M$ for $n\geq 2$. Then $M$ has finite $p$-exponents for any prime $p$ if and only if $M$ is rationally elliptic.\\
\textup{(b)} Consider any connected sum of sphere products $T$. Then $T$ has finite $p$-exponents for any prime $p$ if and only if $T$ is a sphere product if and only if $T$ is rationally elliptic.\\
\textup{(c)} Consider any CW complex $X$ as in Theorem C with $r\geq 3$. Then $X$ is rationally hyperbolic and except for finitely many primes has unbounded $p$-exponents.}
\end{thmd}
\noindent To verify Moore conjecture for a rationally hyperbolic space we ought to exhibit a prime $p$ for which no (finite) exponent exists. In Theorem D (a), (b) the class of rationally hyperbolic manifolds have the stronger property that no $p$-exponent exists for {\it any} prime $p$.\\

In the process of writing this manuscript, we have come across the paper \cite{BeTh14}, where among other results the authors prove part (a) of Theorems A and B in the case $n\neq 4$ or $8$. These are proved using methods that decompose the loop space into simpler factors. Our techniques, being different, also covers the cases $n=4$ and $n=8$. In this respect we directly deduce similar loop space decomposition results in \S \ref{OmM}. These results extend the relevant theorems of \cite{BeTh14} to the case $n=4$ and $n=8$ provided the $n^\textup{th}$ Betti number is at least $2$. \\

\noindent{\bf Organization of the paper} : In \S \ref{quadassLie} we introduce some preliminaries on quadratic associative algebras and Lie algebras. This includes Koszul duality, Poincar\'{e}-Birkhoff-Witt Theorems and Gr\"obner bases. In \S \ref{quadcoh} we apply these properties to cohomology algebras arising from manifolds. In \S \ref{conn} we deduce the main result for $(n-1)$-connected $2n$-manifolds and analyze some consequences. In \S \ref{Appex} we discuss several applications of the methods developed in this paper. In \S \ref{sc4d} we elucidate the special case of simply connected $4$-manifolds. We repeat these techniques for connected sum of sphere products in \S \ref{connsppr} and for certain CW complexes in \S \ref{hgCW}. Finally, in \S \ref{OmM} we prove that the results in \S \ref{conn} and \S\ref{Appex} yield loop space decompositions.   \\

\noindent{\bf Acknowledgements} : The authors would like to thank Professor Jim Davis for asking the question whether the fact that the homotopy groups of simply connected four manifolds are determined by the second Betti number can be deduced by homotopy theoretic methods.

\section{Quadratic Associative Algebras and Lie Algebras}\label{quadassLie}

In this section we introduce some algebraic preliminaries related to Koszul duality of {\it associative algebras} and associated {\it Lie algebras}. We also recall some results and relations between Gr\"obner bases of quadratic algebras, quadratic Lie algebras and, Poincar\'e-Birkhoff-Witt Theorem. These are accompanied by some crucial algebraic results used throughout the manuscript. 

\subsection{Koszul duality of associative algebras}

We begin with some background on Koszul duality of quadratic algebras over a field following \cite{PolPos05} and \cite{LoVa12}. Throughout this subsection $k$ denotes a field, $V$ a $k$-vector space and $\otimes=\otimes_k$ unless otherwise mentioned. 

\begin{defn}
Let $T_k(V)$ denote the tensor algebra on the space $V$. For $R\subset V\otimes_k V$, the associative algebra $A_k(V,R)=T_k(V)/(R)$ is called a {\it quadratic $k$-algebra}. 
\end{defn} 


A quadratic algebra is graded by weight. The tensor algebra $T(V)$ is {\it weight-graded} by declaring an element of $V^{\otimes n}$ to have grading $n$. Since $R$ is homogeneous, the weight grading passes onto $A(V,R):=A_k(V,R)=T_k(V)/(R)$. 

The dual notion leads to quadratic coalgebras. For this note that $T(V)$ has a coalgebra structure by declaring the elements of $V$ to be primitive. When $T(V)$ is thought of as a coalgebra we write it as $T^c(V)$. 
\begin{defn}
For $R\subset V\otimes V$ the {\it quadratic coalgebra} $C(V,R)$ is defined as the maximal sub-coalgebra $C$ of $T^c(V)$ such that $C\to T^c(V) \to V\otimes V/R$ is $0$. The maximal property of $C(V,R)$ implies that if $C$ is weight-graded sub-coalgebra such that the weight $2$ elements are contained in $R$ then $C\subset C(V,R)$. 
\end{defn}

Recall that a $k$-algebra $A$ is {\it augmented} if there is a $k$-algebra map $A\to k$. Analogously a $k$-coalgebra $C$ is {\it coaugmented} if there is a $k$-coalgebra map $k\to C$. For a coaugmented coalgebra $C$ one may write $C \cong k \oplus \bar{C}$ and the projection of $\Delta$ onto $\bar{C}$ as 
$\bar{\Delta} : \bar{C} \to \bar{C} \otimes \bar{C}.$
A coaugmented coalgebra is said to be {\it conilpotent} if for every $c\in \bar{C}$ there exists $r>0$ such that $\bar{\Delta}^r(c)=0$.  

One has adjoint functors between augmented algebras and coaugmented colgebras given by the bar construction and the cobar construction (see \cite{LoVa12}, \S 2.2.8).


\begin{defn}
Let $\bar{A}\subset A$ be the kernel of the augmentation, define $BA= (T(s\bar{A}),d)$, where $s$ denotes suspension and $d$ is generated as a coderivation by 
$$d(s(a))=s(a\otimes a)-s(a\otimes 1) -s(1\otimes a).$$
Dually, let $C=\bar{C}\oplus k$, and define $\Omega C = (T(s^{-1}\bar{C}),d)$ where $d$ is generated as a derivation by the equation 
$$d(s^{-1} c) = s^{-1}(\bar{\Delta} (c))=s^{-1}(\Delta(c)-c\otimes 1 - 1\otimes c).$$
\end{defn}

Note that the above definition makes sense for quadratic algebras (and coalgebras) as these are naturally augmented (respectively coaugmented). There is a differential on $C\otimes \Omega C$ generated by $d(c)=1\otimes s^{-1}c$ and dually a differential on $A\otimes BA$.  
\begin{defn}
The Koszul dual coalgebra of a quadratic algebra $A(V,R)$ is defined as $A^\text{!`}=C(s(V),s^2(R))$. The Koszul dual algebra $A^!$ of a quadratic algebra $A(V,R)$ is defined as $A^!=A(V^*,R^\perp)$ where $R^\perp\subset V^*\otimes V^*$ consists of elements which take the value $0$ on $R\subset V\otimes V$. 
\end{defn}

 The Koszul dual algebra and the Koszul dual coalgebra are linear dual upto a suspension. Let $A^{(n)}$ stand for the subspace of homogeneous $n$-fold products. Then $(A^!)^{(n)}\cong s^n((A^\text{!`})^*)^{(n)}$. 

 For a quadratic algebra $A(V,R)$, there is a natural map from $\Omega A^\text{!`} \rightarrow A$ which maps $v$ to itself. Using this map there is a differential on $A^\text{!`}\otimes_\kappa A$ denoted by $d_\kappa$.      

\begin{defn}{\bf (\cite{LoVa12}, Theorem 3.4.6)}
\,A quadratic algebra $A(V,R)$ is called {\it Koszul} if one of the following equivalent conditions hold:\\
(i) $\Omega A^\text{!`}\rightarrow A$ is a quasi-isomorphism; \\
(ii) the chain complex $A^\text{!`}\otimes_\kappa A$ is acyclic; \\
(iii)\footnote{See  \cite{PolPos05}, Chapter 2, Definition 1.} $Ext_A(k,k)\cong A^!$. 
\end{defn}

Koszulness is an important property of quadratic algebras. One may use this to compute the homology of the cobar construction in various examples. We recall a condition for Koszulness. Fix a basis $(v_1,v_2,\ldots, v_n)$ of $V$, and fix an order $v_1<v_2<\ldots<v_n$. This induces a lexicographic order on the degree $2$ monomials. Now arrange the expressions in $R=\textup{span}_k\{r_1,r_2,\ldots\}$ in terms of order of monomials. An element $v_iv_j$ is called a leading monomial if there exists $r_l=v_iv_j + \mbox{lower order terms}$. Note that (cf. \cite{LoVa12}, Theorem 4.1.1) implies that if there is only one leading monomial $v_iv_j$ with $i\neq j$ then the algebra is Koszul. This leads to the following result.
\begin{prop}\label{Kos2}
Let $V$ be a $k$-vector space and $R=k\mathpzc{r}\subset V\otimes V$ be a $1$-dimensional subspace such that with respect to some basis $\{v_1,\ldots,v_n\}$ of $V$, 
$$\mathpzc{r}=v_iv_j +\sum_{k<i~or~k=i,l<j} a_{k,l} v_kv_l$$
for $i\neq j$ in the sense above. Then the algebra $A=A(V,R)=T(V)/R$ is Koszul.
\end{prop} 

Next we deduce an useful result which is used throughout the paper. Call an expression in $V\otimes V$ symmetric if with respect to some basis the coefficient of $v_i\otimes v_j$ and $v_j\otimes v_i$ are the same\footnote{Note that this condition does not depend on the choice of basis. An element is symmetric if and only if the corresponding quadratic form is symmetric.}. 
\begin{lemma}\label{Kos3}
Let $V$ be a $k$-vector space and $R\subset V\otimes V$ be a $1$-dimensional subspace generated by a symmetric element. Then the algebra $A=A(V,R)$ is Koszul.
\end{lemma}
\begin{proof}
We start by making some reductions. Suppose that a generator $\mathpzc{r}$ of $R$ is of the form $\mathpzc{r}=\sum a_{i,j} v_i\otimes v_j$ with $a_{i,j}=a_{j,i}$. If $a_{n,n}= 0$ and some $a_{i,n}\neq 0$ for $i<n$, then we put the order $v_1<\ldots <v_n$ and then from Proposition \ref{Kos2} it follows that $A(V,R)$ is Koszul. If $a_{i,n}=0$ for all $i$ then we recursively argue with $\{v_1,\ldots,v_{n-1}\}$.

If $a_{n,n}\neq 0$ then we change the basis by $v_i\mapsto v_i$ for $i<n$ and 
 $$v_n \mapsto v_n - \sum_{i<n} \frac{a_{i,n}}{a_{n,n}} v_i.$$ 
In the new basis the generator is of the form
$$\mathpzc{r}=a_{n,n}v_n^2 + \sum_{i,j<n} a'_{i,j}v_i\otimes v_j.$$
If $a'_{n-1,n-1}\neq 0$ then we apply the same argument as above to $v_{n-1}$. Otherwise, if $a'_{i,n-1}=0$ for all $i\leq n-1$ then we argue with $\{v_1,\ldots, v_{n-2}\}$ and proceed inductively. Therefore we are reduced to the case of $\mathpzc{r}=a_1v_1^2+\ldots +a_lv_l^2$.

Next we make the reduction to $l=n$. In the sense of (\cite{PolPos05}, Ch 3, Corollary 1.2) we have that 
$$T(v_1,\ldots,v_n)/(a_1v_1^2+\ldots +a_l v_l^2) \cong T(v_1,\ldots,v_l)/(a_1v_1^2+\ldots + a_lv_l^2) \sqcup T(v_{l+1},\ldots,v_n)$$
where $\sqcup$ stands for the coproduct in the category of associative $k$-algebras. Since $T(v_{l+1},\ldots, v_n)$ is Koszul, in view of  \cite{PolPos05}, Ch 3, Corollary 1.2, it suffices to prove that $T(v_1,\ldots,v_l)/(a_1v_1^2+\ldots + a_lv_l^2)$ is Koszul. Thus, we may assume $l=n$. 

Denote the $k$-algebra $T(v_1,\ldots,v_n)/(a_1 v_1^2 +\ldots + a_n v_n^2)$ by $A$. Note that $a_i\neq 0$ for every $i$. If $n=1$, then the algebra $A$ is just $T(v)/(v^2)$, which is Koszul. So now assume $n\geq 2$. In this case we verify $A$ is Koszul by proving  $\mathit{Ext}_A(k,k)=A^!$. 

Write down a basis for the algebra $A$ comprising of monomials 
$$\{v_{i_1}v_{i_2}\ldots v_{i_l}|~1\leq i_j\leq n, ~ \nexists ~j~\mbox{such that}~i_j=i_{j+1}=n  \}.$$
The above basis written using the lexicographic order coming from $v_1<v_2\ldots <v_n$ turns out to be
$$v_1,\ldots,v_n,v_1^2,v_1v_2,\ldots,v_{n}v_{n-1},v_1^3,\ldots .$$

Note that multiplication by $v_1$ on the right (respectively on the left) carries a basis of weight grading $l$ to distinct basis elements of weight grading $l+1$ ending with $v_1$ (respectively starting with $v_1$). It follows that right multiplication by $v_1$ on $A$ is injective. A similar argument yields that right multiplication by $v_i$ is injective for every $i<n$. In fact one can say slightly more: if $i<n$ then right multiplication by $v_i$ for different $i$ maps basis monomials of weight $l$ to distinct basis monomials of weight $l+1$. Thus, if $\mathbb{B}_l$ is the set of basis monomials of weight $l$, then the sets 
$$\mathbb{B}_lv_1,~ \mathbb{B}_lv_2,\ldots, \mathbb{B}_lv_{n-1}$$
are all disjoint. It follows that for $\alpha_1,\ldots, \alpha_{n-1}\in A$ 
$$\sum_{i<n}\alpha_i v_i = 0 \implies \alpha_1=\alpha_2=\cdots = \alpha_{n-1}=0.$$

 The algebra $A^!$ is $T(V^*)/(R^\perp)$. Taking a dual basis $\{v_1^*,\ldots,v_n^*\}$ of $V^*$, observe that $R^\perp$ contains terms $v_i^*v_j^*$ for $i\neq j$ and $a_j(v_i^*)^2 - a_i(v_j^*)^2$. Note that all weight three monomials of $A^!$ lie in the ideal generated by $R^\perp$; this is clear for monomials not of the form $(v_i^*)^3$ and, for $(v_i^*)^3$ note that
$$a_j(v_i^*)^3=a_i v_i^*(v_j^*)^2=a_i(v_i^*v_j^*)v_j^* \in (R^\perp).$$
Therefore the algebra $A^!$ is concentrated in weight grading $0$, $1$ and $2$. The basis in weight grading $1$ is $\{v_1^*,\ldots,v_n^*\}$; in weight grading $2$ the algebra is $1$ dimensional and for some generator $z^*$ we have 
$$v_i^*v_j^*=0~\forall ~i\neq j,~ (v_i^*)^2 = a_i z^*.$$ 

We write down a resolution of $k$ by free $A$-modules as 
$$P_2\stackrel{d_2}{\longrightarrow} P_1\stackrel{d_1}{\longrightarrow} P_0\longrightarrow k$$ 
where 
$$P_0=A,~P_1=A^{\oplus n},~P_2=A$$
and for $\alpha_i\in A$
$$d_1(\alpha_1,\ldots,\alpha_n) = \sum \alpha_i v_i,~ d_2( \alpha) = (a_1\alpha v_1, \ldots, a_n \alpha v_n).$$
It is clear that $d_1\circ d_2 (\alpha)= \alpha \big(\sum a_i v_i^2\big)=0$ so that the above is a complex. To prove that this is a resolution we need to check that $ \mathit{Ker}(d_1)\subset \mathit{Im}(d_2)$ and $\mathit{Ker}(d_2)=0$. The latter condition follows from the fact that right multiplication by $v_i$ is injective. 

Note that the $d_i$ preserve grading if we write $P_2=\Sigma^2 A$, $P_1=(\Sigma A)^{\oplus n}$ and $P_0=A$ (with respect to weight grading on $A$). Therefore it is enough to verify $\mathit{Ker}(d_1) \subset  \mathit{Im}(d_2)$ on homogeneous pieces.  Let $(\alpha_1,\ldots,\alpha_n)$ be an element of $\mathit{Ker}(d_1)$ so that the weight of each $\alpha_i$ is $l$. Write each $\alpha_i$ as a linear combination of the basis elements described above. Consider a monomial $v_{i_1}\ldots v_{i_l}$ in $\alpha_n$. Under $d_1$ this goes to $ v_{i_1} \ldots v_{i_l} v_n$. This is linearly independent from monomials that occur in $\alpha_i v_i$ unless $i_l=n$. It follows that $\alpha_n = a_n\alpha v_n$ (the factor $a_n$ can always be introduced as it is a non-zero element of the base field). Hence we have 
$$d_1(\alpha_1,\ldots,\alpha_n)= \sum_{i<n} \alpha_iv_i + a_n \alpha v_n^2 = \sum_{i<n} (\alpha_i -a_i\alpha v_i)v_i \implies \alpha_i=a_i\alpha v_i ~\forall ~ i.$$
Therefore $(\alpha_1,\ldots,\alpha_n)=d_2(\alpha)$ proving that the sequence above is indeed a resolution. Applying $\mathit{Hom}_A(-,k)$, we may compute $\mathit{Ext}_A(k,k)\cong A^!$. 
\end{proof}

\subsection{Lie algebras and quadratic algebras}
Let $\mathcal{R}$ be a Principal Ideal Domain (PID). We recall some facts about quadratic algebras and Lie algebras over $\mathcal{R}$. In this manuscript the domains used will be $\Z$ or some localization of $\Z$. Suppose $V$ is a free $\mathcal{R}$-module that is finitely generated and $R\subset V\otimes_\RR V$ be a submodule. As before, denote by $A(V,R)$ the quadratic $\RR$-algebra $T_\RR(V)/(R)$ where $T_\RR(V)$ is the $\RR$-algebra tensor algebra on $V$ and $(R)$ is the two-sided ideal on $R$.  

Recall the Diamond Lemma from \cite{Berg78}. Suppose that $V$ has a basis $x_1,\ldots,x_m$. Suppose that the submodule $R\subset V\otimes V$ is generated by $n$ relations of the form $W_i = f_i$ where $W_i=x_{\alpha(i)} \otimes x_{\beta(i)}$ and $f_i$ a linear combination of terms $x_j\otimes x_l$ other than $x_{\alpha(i)}\otimes x_{\beta(i)}$. We call a monomial $x_{i_1}\otimes \ldots \otimes x_{i_l}$ {\it $R$-irreducible} if it cannot be expressed as $A\otimes x_{\alpha(i)}\otimes x_{\beta(i)} \otimes B$ for monomials $A,B$ for any $i$. Theorem 1.2 of \cite{Berg78} states that under certain conditions on $R$, the images of the $R$-irreducible monomials in $A(V,R)$ forms a basis. One readily observes that the conditions are satisfied if $n=1$ and $\alpha(1) \neq \beta(1)$.  Therefore we conclude the following theorem.
\begin{theorem}\label{Diamond}
Suppose that $R=k\mathpzc{r}$ and $\mathpzc{r}$ is of the form 
$$x_\alpha\otimes x_\beta = \sum_{(i,j)\neq (\alpha, \beta)} a_{i,j} x_i \otimes x_j$$ 
with $\alpha\neq \beta$. Then the $R$-irreducible elements form a basis for $A(V,R)$. 
\end{theorem}

Recall that elements in $V\otimes V$ are symmetric if the corresponding quadratic form is symmetric. Similarly define anti-symmetric elements as those for which the corresponding quadratic form is skew-symmetric and the diagonal entries are zero. Define non-singular elements as those for which the corresponding quadratic form is non-singular. We prove the following proposition.

\begin{prop}\label{freemod}
Consider a quadratic algebra $A(V,R)$ over a Principal Ideal Domain $\RR$ such that $R\subset V\otimes_\RR V$ is a free $\RR$-module of rank $1$ generated by a non-singular, symmetric element (or an anti-symmetric element). Then, 
\begin{enumerate}
\item As a $\RR$-module the quadratic algebra $A(V,R)$ is free. 
\item If $R$ is generated by an anti-symmetric element or if $\dim(V)\neq 2$ or if $\RR=\Z$ or a local ring, then there is a generator of $R$ of the form $W=f$ where $W$ is a monomial and $f$ does not contain $W$. 
\item Under the above assumptions the $R$-irreducible elements form a basis. 
\end{enumerate}
\end{prop}
\begin{proof}
Our method relies on the variant of Diamond Lemma as presented in Theorem \ref{Diamond}. If $V=\RR\{a_1,\ldots,a_r\}$ then $R$ is generated by an element of the form $\sum g_{i,j} a_i \otimes a_j$. The element is symmetric (respectively anti-symmetric) if the corresponding quadratic form is symmetric (respectively skew-symmetric). 

The proof for an anti-symmetric element is easier. The non-singular anti-symmetric element induces a symplectic inner product on $V$ (cf. \cite{MiHu73}, Ch 1, Definition 1.3). Any symplectic inner product space over a Dedekind Domain possesses a symplectic basis by (\cite{MiHu73}, Ch 1, Corollary 3.5). Therefore the relation can be written as 
$$a_1a_2 = a_2a_1 + \mathit{terms~not~involving~}a_1\mathit{~and~}a_2.$$
Therefore by Theorem \ref{Diamond}, $A(V,R)$ possesses a basis satisfying the required criteria. 

For the symmetric case one needs to do a little work to prove that the hypothesis of Theorem \ref{Diamond} is satisfied. It suffices to express the element in the form 
$$a_1a_2=  \mathit{linear~combinations~of~terms~other~than~}a_1a_2.$$
Let $\beta$ be the quadratic form and $Q$ be the fraction field of $\RR$. Suppose $\dimn(V)=r\geq 3$.  Due to dimensional constraints there exists a $Q$-linear combination $v$ of $a_2,\ldots,a_r $ such that $\beta(v,a_1)=0$. Clear out denominators of $v$ so that $v$ is a primitive combination of $a_2,\ldots,a_r$. Hence one may change the basis $a_2,\ldots,a_r$ to $v,b_3 \ldots,b_r$. Thus we may assume $\beta(a_1,a_2)=0$. Now as the bilinear form is invertible over $\RR$ there exists $w$ such that $\beta(a_1,w)=1$. Note this computation does not change if we add multiples of $a_2$. Therefore in terms of the basis we may assume that the coefficient of $a_2$ in $w$ is $1$. We may now replace $a_2$ in the basis by $w$ and obtain a new basis with the property $\beta(a_1,a_2)=1$. This basis satisfies the required condition. 

When $r=1$ then $R=(a_1^2)$ and $A(V,R)\cong\RR$ is free. Finally, we consider $r=2$. First consider the case when $\RR=\Z$. From  Ch 2, Theorem 2.2 of \cite{MiHu73} we know that such binary forms either possess an orthogonal basis or is hyperbolic. In either case, there is a basis $\{v_1,v_2\}$ such that $\beta(v_1,v_2)=1$ and the result follows. 

Next suppose that $\RR$ is a local ring, so that we have $R$ is generated by $a a_1^2 +ba_1a_2 +ba_2a_1+d a_2^2$ where $ad-b^2$ is a unit. Suppose the maximal ideal in $\RR$ is of the form $(\pi)$. It follows that $\pi$ cannot divide all of the elements $a,b,d$ thus at least one of them is a unit. If $a$ or $d$ is a unit, the form possesses an orthogonal basis. If $b$ is a unit changing the basis we may arrange $b=1$ so that there is a basis $\{v_1,v_2\}$ so that $\beta(v_1,v_2)=1$. Again the result follows in either case.


Finally consider the case of a general PID $\RR$. We only need to show that $A(V,R)$ is free. The algebra $A(V,R)$ is weight-graded and finitely generated in each degree. Localizing $A(V,R)$ at each prime $\pi$ of $\RR$ gives a free module by the above argument. Thus, $A(V,R)$ is free.
\end{proof} 

\mbox{ }

For a quotient field $k$ of the PID $\RR$, the quadratic algebra with the corresponding basis and relation is obtained by tensoring with $k$. 
\begin{prop}\label{irr}
Let $V,R$ be as in Proposition \ref{freemod} and let $\pi\in \RR$ be an irreducible element, so that $k=\RR/\pi$ is a field. Then 
$$A_\RR(V,R)\otimes_\RR k  \cong A_k(V\otimes_\RR k,R\otimes_\RR k).$$ 
\end{prop}

\begin{proof}
The construction of the quadratic algebras induces a map 
\begin{equation}
A_\RR(V,R)\otimes_\RR k  \to A_k(V\otimes_\RR k,R\otimes_\RR k).\label{quadbasech}
\end{equation}
The proof of Proposition \ref{freemod} implies that $A_\RR(V,R)$ and $A_k(V\otimes_\RR k, R\otimes_\RR k)$ satisfy the conditions of Theorem \ref{Diamond}. It follows that the map is an isomorphism as the basis from Theorem \ref{Diamond} of the left hand side of \eqref{quadbasech} is mapped bijectively onto the basis of the right hand side.
\end{proof}

One may make an analogous construction of a Lie algebra in the case $R\subset V\otimes V$ lies in the free Lie algebra generated by $V$. Construct the quadratic Lie algebra $L(V,R)$ as $L(V,R)=\frac{\Lie(V)}{(R)}$, where $\Lie(V)$ is the free Lie algebra on $V$ and $(R)$ is the Lie algebra ideal generated by $R$. We readily prove the following useful result. 


\begin{prop}\label{univlie}
The universal enveloping algebra of $L(V,R)$ is $A(V,R)$. 
\end{prop}

\begin{proof}
This follows from the universal property of universal enveloping algebras. One knows that the universal enveloping algebra of the free Lie algebra $L(V)=L(V,0)$ is the tensor algebra $T_\RR(V)=A(V,0)$. Consider the composite 
$$L(V)\to T_\RR(V)\to T_\RR(V)/(R).$$
This maps $R$ to $0$. Hence we obtain a map $L(V,R)= \Lie(V)/(R) \to A(V,R)$.

Now suppose we have a map $L(V,R)\to A$ where $A$ is an associative algebra. This means we have an arbitrary map from $V\to A$ so that $R$ is mapped to $0$. Hence it induces a unique map from $A(V,R)\to A$. Thus $A(V,R)$ satisfies the universal property of the universal enveloping algebra.   
\end{proof} 

Recall the Poincar\'e-Birkhoff-Witt Theorem for universal enveloping algebras of Lie algebras over a commutative ring $\RR$. For a Lie algebra $\LL$ over $\RR$ which is free as a $\RR$-module, the universal enveloping algebra of $\LL$ has a basis given by monomials on the basis elements of $\LL$. We may use this to deduce the following result.

\begin{prop}
\label{PBW-quad}
Suppose $V$ and $R$ are as in Proposition \ref{freemod}. Then the quadratic algebra $A(V,R)$ is isomorphic to the symmetric algebra on the $\RR$-module $L(V,R)$.  In terms of the multiplicative structure, the symmetric algebra on $L(V,R)$ is isomorphic to the associated graded of $A(V,R)$ with respect to the length filtration induced on the universal enveloping algebra.  
\end{prop}

In the above Proposition the length filtration on $U(L)$ for a Lie algebra $L$ is defined inductively by 
$$F_0 U (L) = R$$
and
$$F_nU (L) = F_{n-1}U (L) + \mathit{Im}\{ L \otimes F_{n-1}U (L) \to U (L) \otimes U (L) \to U (L)\}.$$

\subsection{Bases of quadratic Lie algebras} 
For quadratic Lie algebras $L(V,R)$ over a PID $\RR$ as above, one may write down an explicit basis using methods similar to the Diamond Lemma. 

Fix $V$, a free $\RR$-module, and $R$ an anti-symmetric element of $V\otimes V$ of the form 
$$a_1\otimes a_2= a_2\otimes a_1 +  \mathit{terms~not~involving~}a_1\mathit{~and~}a_2.$$
One may conclude this about $R$ if it is generated by a non-singular element as in Proposition \ref{freemod}. The anti-symmetry condition ensures that $R$ lies in the free Lie algebra generated by $V$. The proof of Proposition \ref{freemod} shows that for such $V$ and $R$ one obtains a basis of $A(V,R)$ as prescribed in the Diamond Lemma.

Consider the quadratic Lie algebra $L(V,R)$ for $R$ as the above. As a Lie algebra element $R$ has the form  
$$[a_1, a_2]=  \mathit{terms~not~involving~}a_1\mathit{~and~}a_2.$$

\begin{prop}
The $\RR$-module $L(V,R)$ is free.
\end{prop}

\begin{proof}
From Proposition \ref{univlie} we obtain that the enveloping algebra of $L(V,R)$ is $A(V,R)$ which is free by Proposition \ref{freemod}. Over a Dedekind Domain the map from the Lie algebra to its universal enveloping algebra is an inclusion (\cite{Car58},\cite{Laz54}; also see \cite{Cohn63}). It follows that the map $L(V,R)$ to $A(V,R)$ is an inclusion.

 Introduce the weight grading on $L(V,R)$ and $A(V,R)$ by declaring the grading of $V$ to be $1$. Then the inclusion of $L(V,R)$ in $A(V,R)$ is a graded map. Now each graded piece of $A(V,R)$ is free and finitely generated and the corresponding homogeneous piece of $L(V,R)$ is a submodule. This is free as $\RR$ is a PID.   
\end{proof}

Denote by $L^{w}(V,R)$ (respectively $A^w(V,R)$) the homogeneous elements of $L(V,R)$ (respectively $A(V,R)$) of weight $w$. Recall from \cite{LaRam95} and \cite{Loth97} the concept of a Lyndon basis of a Lie algebra described by generators and relations.

\begin{defn}
Define an order on $V$ by $a_1<a_2<\ldots <a_r$. A {\it word} (in the ordered alphabet $\{a_1,\ldots,a_r\}$) is a monomial in the $a_i$'s. The number of alphabets in a word $w$ is denoted by $|w|$ and called the {\it length} of the word. The lexicographic ordering on the words will be denoted by $>$, i.e., $w> \mathpzc{w}$ means that the word $w$ is lexicographically bigger than $\mathpzc{w}$.

A {\it Lyndon word} is one which is lexicographically smaller than its cyclic rearrangements. Fix an order on all words by declaring words $w\, \preceq\, \mathpzc{w}$ if $|w|\leq |\mathpzc{w}|$, and if $|w|=|\mathpzc{w}|$ then $w\geq \mathpzc{w}$ in lexicographic order. 
\end{defn}

Let $L$ be the set of Lyndon words. Any Lyndon word $\mathpzc{l}$ with $|\mathpzc{l}|>1$ can be uniquely decomposed into $\mathpzc{l}=\mathpzc{l}_1\mathpzc{l}_2$ where $\mathpzc{l}_1$ and $\mathpzc{l}_2$ are Lyndon words so that $\mathpzc{l}_2$ is the larget proper Lyndon word occuring from the right in $\mathpzc{l}$. Each Lyndon word $\mathpzc{l}$ is associated to an element $b(\mathpzc{l})$ of the free Lie algebra on $V$. This is done inductively by setting $b(a_i)=a_i$ and for the decomposition above 
$$b(\mathpzc{l})=b(\mathpzc{l}_1\mathpzc{l}_2):=[b(\mathpzc{l}_1),b(\mathpzc{l}_2)].$$
The image, under $b$, of the set of Lyndon words form a basis of $\Lie(a_1,\ldots,a_r)$, the free Lie algebra on $a_1,\ldots ,a_r$ over any commutative ring (\cite{Loth97}, Theorem 5.3.1). 

Let $J$ be a Lie algebra ideal in $\Lie(a_1,\ldots,a_r)$. Denote by $I$ the ideal generated by $J$ in $T(a_1,\ldots,a_r)$. Define a Lyndon word $\mathpzc{l}$ to be {\it $J$-standard} if $b(\mathpzc{l})$ cannot be written as a linear combination of strictly smaller (with respect to $\preceq$) Lyndon words modulo $J$. Let $S^JL$ be the set of $J$-standard Lyndon words. Similarly define the set $S^I$ of $I$-standard words with respect to the order $\preceq$. We will make use of the following results.
\begin{prop} \label{SL=L} {\bf (\cite{LaRam95}, Corollary 2.8)}
With $I,J$ as above we have $L \cap S^I = S^JL$.  
\end{prop}
\begin{prop}\label{SLbasis}{\bf (\cite{LaRam95}, Theorem 2.1)}
The set $S^JL$ is a basis of $\Lie(a_1,\ldots,a_r)/ J$ over $\Q$. 
\end{prop}

We restrict our attention to $J=(R)$ where $R$ has the form described above. Let $S_RL$ be the set of all Lyndon words which does not contain $a_1a_2$ consecutively in that order. From the Diamond Lemma, $S^I$ is the same as the $R$-irreducible monomials which are the monomials not of the form $Aa_1a_2B$. It follows that $S^{(R)}L=S^I \cap L = S_RL$.  We prove the following theorem.
\begin{theorem}\label{Liebasis}
Suppose $\RR$ is a localization of $\Z$. Then the images of $b(S_RL)$ form a basis of $L(V,R)$.
\end{theorem}

\begin{proof}
Proposition \ref{SLbasis} implies the Theorem for $\RR=\Q$. We use it to obtain the result for $\RR$, a localization of $\Z$. First check that the images must indeed generate the Lie algebra $L(V,R)$. Any Lyndon word outside $S^{(R)}L$ can be expressed using lesser elements and continuing in this way we must end up with elements in $S^{(R)}L$. Therefore $S^{(R)}L$ spans $L(V,R)$ and as $S^{(R)}L=S_RL$ the latter also spans. 

It remains to check that $S_RL$ is linearly independent. From Proposition \ref{SLbasis}, $S_RL$ is linearly independent over $\Q$ the fraction field of $\RR$. It follows that it must be linearly independent over $\RR$.
\end{proof}

\subsection{Graded Lie algebras}

We recall the Poincar\'{e}-Birkhoff-Witt Theorem for graded Lie algebras over commutative rings from \cite{Nei10}. A graded Lie algebra over a ring $\RR$ in which $2$ is not invertible carries an extra squaring operation on odd degree classes to encode the relation $x^2 = \frac{1}{2} [x,x]$ whenever $|x|$ is odd. 
\begin{defn}
 A graded Lie algebra $L=\oplus L_i$ is a graded $\RR$-module together with a Lie bracket
$$[~ ,~ ] : L_i \otimes_{\RR} L_j \to  L_ {i+j}$$
and a quadratic operation called {\it squaring} defined on odd degree classes
$$(~)^2 : L_{2k+1} \to L_{4k+2}.$$
These operations are required to satisfy the identities
\begin{align*}
[x, y] & = -(-1)^{\deg(x)\deg(y)}[y, x]\,\,\,(x\in L_i, y\in L_j)\\
[x, [y, z]] & = [[x, y], z] + (-1)^{\deg(x)\deg(y)} [y, [x, z]] \,\,\,(x\in L_i, y\in L_j, z\in L_k)\\
(ax)^ 2 & = a^2 x^2 \,\,\,(a\in \RR, x\in L_{2k+1})\\
(x + y)^2 & = x^2 + y^2 + [x, y] \\
[x, x] & = 0 \,\,\,(x\in L_{2i})\\
2x^2 & = [x, x],\,\,[x, x^2] = 0 \,\,\,\hfill(x\in L_{2k+1})\\
[y, x^2] & = [[y, x], x] \,\,\,(x\in L_{2k+1}, y\in L_i).
\end{align*}
\end{defn}

One obtains analogous constructions of universal enveloping algebras of graded Lie algebras and free graded Lie algebras. The universal enveloping algebra of a graded Lie algebra is also a graded associative algebra.  As a graded algebra the universal enveloping algebra of the free graded Lie algebra on a graded $A$-module $V$ free in each dimension is the tensor algebra $T_A(V)$ (\cite{Nei10}, Example 8.1.4). 

Analogously one may define quadratic graded $A$ algebras $A^{gr}(V,R)$ and quadratic graded Lie algebras $L^{gr}(V,R)$ when the graded subspace $R\subset V\otimes V$ is contained in the free graded Lie algebra generated by $V$. One readily proves a graded version of  Proposition \ref{univlie}.

\begin{prop}
The universal enveloping algebra of $L^{gr}(V,R)$ is $A^{gr}(V,R)$. 
\end{prop}

It is possible to derive a Poincar\'e-Birkhoff-Witt theorem for graded Lie algebras under the extra assumption that the underlying module is free over $\RR$. 

\begin{theorem}\label{PBW} {\bf (\cite{Nei10}, Theorem 8.2.2)} If $L$ is a graded Lie algebra over $\RR$  which is a free $\RR$-module in each degree, then $U (L)$ is isomorphic to the symmetric algebra on $L$. In terms of the multiplicative structure, the symmetric algebra on $L$ is isomorphic to the associated graded of $U(L)$ with respect to the length filtration\footnote{Refer to the discussion immediately following Proposition \ref{PBW-quad}.} induced on $U(L)$. 
\end{theorem}

If the base ring $\RR$ is a PID, then we deduce that any graded Lie algebra embeds in its universal enveloping algebra. 

\begin{theorem}
Suppose that $\RR$ is a Principal Ideal Domain. Let $L$ be a graded Lie algebra such that $L_n$ is finitely generated for every $n$. Let $U(L)$ be its universal enveloping algebra. Then the natural map $\iota:L \to U(L)$ is injective. 
\end{theorem}

\begin{proof}
Let $\iota$ denote the map $L\to U(L)$. Let $\pi$ be a prime in $\RR$ and $k$ be the field $\RR/\pi$. Note that the universal enveloping algebra of the Lie algebra $L\otimes_\RR k$ over $k$ is isomorphic to $ U(L)\otimes_\RR k$. It follows from Theorem \ref{PBW} that the map 
 $$\iota\otimes 1:L\otimes_\RR k \to U(L)\otimes_\RR k$$ 
is injective. Therefore we obtain for any prime $\pi$ of $\RR$, $\mathit{Ker}(\iota) \otimes_\RR \RR/\pi=0$. As $\mathit{Ker}(\iota)_n$ is finitely generated over the PID $\RR$, it follows that   $\mathit{Ker}(\iota)_n=0$.
\end{proof}


\section{Quadratic Properties of Cohomology Algebras}\label{quadcoh}
In this section we consider the cohomology ring of an $(n-1)$-connected $2n$-manifold $M$. These algebras will be put to use in the following sections. We prove that when the rank of $H^n(M)$ is at least $2$ the cohomology ring is a quadratic algebra which is Koszul. This result will be used frequently in later sections.

Let $M$ be a closed manifold of dimension $2n$ which is $(n-1)$-connected ($n\geq 2$). The cohomology of $M$ has the form 
$$H^i(M)=\left \{\begin{array}{rl} 
                \Z &\mbox{if}~i=0,2n \\
                \Z^r &\mbox{if}~i=n  \\
                0    &\mbox{otherwise}.                   
\end{array}\right.$$  
The multiplication is graded-commutative; for $\alpha,\beta \in H^n(M)$
$$\alpha \beta =\left \{\begin{array}{rl} 
                \beta \alpha &\mbox{if $n$ is even} \\
                -\beta \alpha &\mbox{if $n$ is odd}.                   
\end{array}\right.$$

\begin{prop}\label{quadratic}
Let $k$ be any field. $H^*(M;k)$ is a quadratic algebra if $n$ is odd. If $n$ is even and $r\geq 2$, then $H^*(M;k)$ is a quadratic algebra. 
\end{prop}

\begin{rmk}\label{Hopf}
If $r=1$, then the cohomology ring of $M$ is of the form $\Z[x]/(x^3)$ which is not quadratic. However, for this ring to be the cohomology of a manifold, there must exist elements of Hopf invariant one in $\pi_{2n-1}(S^n)$ which means $n\in \{1,2,4,8\}$.  
\end{rmk}

In view of Remark \ref{Hopf}, Proposition \ref{quadratic} follows from the more general Proposition \ref{cellquadratic} below. For any space $X$ and field $k$ we say that $X$ is a {\it cohomology $(n-1)$-connected $2n$-manifold over $k$} if 
$$H^i(X;k)=\left \{\begin{array}{rl} 
                k &\mbox{if}~i=0,2n \\
                k^r &\mbox{if}~i=n  \\
                0    &\mbox{otherwise}                   
\end{array}\right.$$
and the bilinear form $H^n(X;k)\otimes H^n(X;k) \to k$ given by the cup product is non-degenerate\footnote{Note that this does not depend on the choice of generator of $H^{2n}(X;k)$.}.

\begin{prop}\label{cellquadratic}
Let  $X$ be a cohomology $(n-1)$-connected $2n$-manifold over $k$. If $n$ is odd and $\mathit{char}(k)\neq 2$, or if $r\geq 2$, then $H^*(X;k)$ is a quadratic algebra. 
\end{prop}

\begin{proof}
We argue differently for the cases $n$ odd and $n$ even. First suppose $n$ is odd and $\mathit{char}(k)\neq 2$. In this case the bilinear form is skew-symmetric and non-degenerate, and hence symplectic. Therefore the dimension of $H^n(X;k)$ is even (that is, $r=2l$) and there is a symplectic basis $\{v_1,w_1,\ldots, v_l,w_l\}$. Let $V=H^n(X;k)$ and $R$ be the subspace of $V\otimes V$ generated by 
$$\{v_i\otimes v_i, ~ w_i\otimes w_i ,~v_i \otimes w_i+w_i \otimes v_i , v_i\otimes v_j,~w_i\otimes  w_j,~v_i\otimes  w_j,~w_j\otimes v_i,v_i\otimes w_i - v_j\otimes  w_j~\mbox{for}~i\neq j\}.$$   
We verify that the quadratic algebra $A=A(V,R)=T(V)/(R)$ is isomorphic to the ring $H^*(X;k)$. Observe that the elements of $R$ are indeed $0$ in $H^*(X;k)$ and hence one has a ring map $\phi:A\to H^*(X;k)$ which sends $V\to H^n(X;k)$ by the identity. We now observe that $\phi$ is an isomorphism. Note that the vector space $V\otimes V/(R)$ is indeed one-dimensional generated by $v_iw_i=-w_iv_i=v_jw_j$. Therefore $\phi$ is an isomorphism between the weight $2$ part of $A$ and $H^{2n}(X;k)$. It remains to check that the weight $3$ part of $A$ is zero. Consider a monomial $abc$ where $a,b,c \in \{v_1,w_1,\ldots, v_l,w_l\}$. Either $ab\in R$ or $bc\in R$ unless $a=v_i,~b=w_i,~c=v_i$ or $a=w_i,~b=v_i,~c=w_i$ for some $i$. In such a case we have the following equations in $A$ 
$$v_iw_iv_i= - w_iv_iv_i=0,~ w_iv_iw_i=-w_iw_iv_i=0.$$
Hence $abc\in (R)$ and $\phi$ is an isomorphism. 

Now consider the case $n$ even  and $\mathit{char}(k)\neq 2$. We proceed similarly assuming $r>1$.  The intersection form is symmetric and non-degenerate. Therefore $V=H^n(X;k)$ has an orthogonal basis $\{v_1,\ldots, v_r\}$  with $ v_i^2=a_i z$ for a chosen generator $z\in H^{2n}(X;k)$ (cf. \cite{MiHu73}, pp. 6, Cor 3.4). We let  $R$ be the subspace of $V\otimes V$ generated by 
$$\{v_iv_j, a_jv_i^2 - a_iv_j^2~\mbox{for}~i\neq j\}.$$   
We again verify that the quadratic algebra $A=A(V,R)$ is isomorphic to $H^*(X;k)$. There is a ring map $\phi:A\to H^*(X;k)$ which sends $V\to H^n(X;k)$ by the identity and $\phi$ is an isomorphism between the weight $2$ part of $A$ and $H^{2n}(X;k)$. Consider a monomial $abc$ where $a,b,c \in \{v_1,\ldots, v_r\}$. Either $ab\in R$ or $bc\in R$ unless $a=b=c=v_i$  for some $i$. Now we use $r\ge 2$ so $\exists~ j \neq i$  and
$$a_jv_iv_iv_i= a_i v_jv_jv_i=0.$$
Hence $abc\in (R)$ and $\phi$ is an isomorphism. 

We are left with the case $\mathit{char}(k)= 2$ and $r\geq 2$. In this case the bilinear form is symmetric and non-degenerate. It follows that the form is the direct sum of a symplectic form and a form which possesses an orthogonal basis (cf. \cite{MiHu73}, pp. 6, Cor 3.3). If the symplectic (respectively orthogonal) part is $0$, then we may repeat the argument when $n$ is even (respectively $n$ odd case) verbatim. In general one has a combination of the two and again $R$ forms a codimension $1$ subspace of $V\otimes V$. The proof follows by using $r \ge 2$ to verify, as before, that there are no non-zero elements of weight at least $3$ in $A(V,R)$. 
\end{proof}

From Proposition \ref{cellquadratic} we record the $V_X$ and $R_X$ so that $H^*(X;k)= A(V_X, R_X)$ for a cohomology $(n-1)$-connected $2n$-manifold $X$ over $k$. We have $V_X=H^n(X;k)$ and  $R_X$ is the kernel $V_X\otimes V_X\to k$ given by $v\otimes w \mapsto a$ where  $v\cup w= a z$. Next we turn our attention to the coalgebra $C(V_X^*,R_X^\perp)$.

The diagonal map on the chain level induces $\Delta:H_*(X)\to H_*(X\times X)$. The target is not quite $H_*(X)\otimes H_*(X)$ and K\"{u}nneth formula gives $H_*(X)\otimes H_*(X) \to H_*(X\times X)$ which evidently is in the wrong direction. However, in case the groups $Tor(H_i(X),H_j(X))$ vanish or with coefficients in a field $k$, the map above is an isomorphism and $H_*(X)$ is a coalgebra. 



\begin{prop}\label{quad-coalg}
Let $k$ be any field and $X$ a  cohomology $(n-1)$-connected $2n$-manifold over $k$. Then  as a coalgebra,  
$$H_*(X;k)\cong C(V_X^*,R_X^\perp).$$
\end{prop}

\begin{proof}
Note that the coalgebra structure on the homology with field coefficients is dual to the ring structure on cohomology. Therefore the coalgebra structure on the homology is the linear dual of the algebra structure on cohomology which in turn is $A(V_X,R_X)$. It follows that the linear dual is the quadratic coalgebra $C(V_X^*,R_X^\perp)$.
\end{proof}

We end this section proving that the quadratic algebras above are Koszul.

\begin{prop}\label{man-Kos}
Let $k$ be a field and $X$ a cohomology $(n-1)$-connected $2n$-manifold over $k$ with $\dim H^n(X;k)\ge 2$. Under these assumptions $H^*(X;k)$ is isomorphic to the quadratic algebra $A(V_X,R_X)$ which is Koszul.
\end{prop}

\begin{proof}
We only need to verify Koszulness. First suppose that $n$ is even and $\mathit{char}(k) \neq 2$. From the proof of Proposition \ref{cellquadratic} there is a basis $\{v_1,\ldots,v_r\}$ of $V_X$ so that $R_X$ is the subspace of $V_X\otimes V_X$ spanned by 
$$\{v_iv_j, a_jv_i^2 - a_iv_j^2~\mbox{for}~i\neq j\}.$$
We deduce $R_X^\perp$ is one-dimensional spanned by $\sum a_i(v_i^*)^2$ which is a symmetric element.  It follows from Lemma \ref{Kos3} that the algebra $A(V_X^*,R_X^\perp)$ is Koszul. This is the Koszul dual algebra of $A(V_X,R_X)$ (the ungraded dual) and therefore, we deduce $A(V_X,R_X)$ is Koszul. 

Now suppose $n$ is odd. Again we prove that $A(V_X^*,R_X^\perp)$ is Koszul. From the proof of Proposition \ref{cellquadratic} we have that there is a basis $\{v_1,w_1,\ldots, v_l,w_l\}$ of $V_X$ so that $R_X$ is spanned by 
$$\{v_i\otimes v_i, ~ w_i\otimes w_i ,~v_i \otimes w_i+w_i \otimes v_i , v_i\otimes v_j,~w_i\otimes  w_j,~v_i\otimes  w_j,~w_j\otimes v_i,v_i\otimes w_i - v_j\otimes  w_j~\mbox{for}~i\neq j\}.$$   
We deduce $R_X^\perp$ is one-dimensional generated by $\mathpzc{r}=\sum (v_i^*w_i^* - w_i^*v_i^*)$. We can introduce an order $v_1^*<w_1^*<\cdots < v_l^*<w_l^*$ for which the element $\mathpzc{r}$ becomes $(-w_l^*v_l^*) + \mbox{lower order terms}$. Thus by Proposition \ref{Kos2} the quadratic algebra $A(V_X^*,R_X^*)$ is Koszul. Hence $A(V_X,R_X)$ is Koszul. 

If $\mathit{char}(k) = 2$, then the bilinear form is a direct sum of a symplectic form and a subspace that possesses an orthogonal basis. If the symplectic summand is non-zero  choose a symplectic basis and sum it with an orthogonal basis. Thus there is a basis $\{v_1,\ldots,v_l\}$ of $V_X$ so that $v_l^2=0$ and $v_lv_{l-1}\neq 0$ ($v_l$ comes from the symplectic summand). Then as above we introduce the order $v_1^*<\ldots<v_l^*$ and $R_X^\perp$ is generated by an element of the form $v_l^*v_{l-1}^*+\mbox{lower order terms}$. It follows that $A(V_X,R_X)$ is Koszul. Otherwise the bilinear form has an orthogonal basis and we repeat the argument of the case of $n$ even verbatim. Hence $A(V_X,R_X)$ is Koszul. 

\end{proof}

\section{Homotopy Groups of $(n-1)$-connected $2n$-manifolds}\label{conn}

In this section we compute the homotopy groups of any closed $(n-1)$-connected $2n$-manifold. The formula expresses the homotopy groups of $M$ in terms of the homotopy groups of spheres in an expression similar to \cite{Hil55}.

The main technique involves computing homology of the loop space $\Omega M$. We express $H_*(\Omega M)$ as the universal enveloping algebra of a certain Lie algebra $\LL_r^u(M)$. The relevant Lie algebra possesses a countable basis $l_1,l_2,\ldots$ which correspond to mapping spheres into $M$ via iterated Whitehead products. We apply the Poincar\'{e}-Birkhoff-Witt Theorem (over $\Z$) to finish off the computation. 


\subsection{Homology of the loop space}
The homology of $M$ is given by the formula 
$$H_i(M)\cong\left \{\begin{array}{rl} 
                \Z &\mbox{if}~i=0,2n \\
                \Z^r &\mbox{if}~i=n  \\
                0    &\mbox{otherwise}.                   
\end{array}\right.$$  
It follows from the Hurewicz theorem that $\pi_n(M)\cong \Z^r$. Let $\alpha_i$ ($1\leq i \leq r$) denote the $r$ inclusions of $S^n \stackrel{\alpha_i}{\longrightarrow} M$ inducing the basis given by the above isomorphism. Then $M$ has a CW complex structure with $r$ $n$-cells and one $2n$-cell attached along an element $[\Lambda] \in \pi_{2n-1}(M)$ which can be expressed as a linear combination of  Whitehead products of $[\alpha_i,\alpha_j]$ ($i<j$) and some classes $\eta_i\in \pi_{2n-1}(S^n)$. 
  
Denote by $a_i$ the Hurewicz images of $\alpha_i$ in $H_*(\Omega M;\Z)$, and by $[\alpha_i]$ the Hurewicz images in $H_*(M)$. Using the Pontrjagin ring structure on $H_*(\Omega M)$ we get a ring map 
$$T_\Z(a_1,\ldots,a_r)\rightarrow H_*(\Omega M).$$
We use the notation in the previous section. Assume that both $n=1,2,4,8$ and $r=1$ does not occur, so that $H^*(M) \cong A(V_M,R_M)$ is a quadratic algebra and $H_*(M)= C(V_M^*,R_M^\perp)$ is a quadratic coalgebra which is Koszul by Propositions \ref{quad-coalg} and \ref{man-Kos}.

\begin{theorem}\label{manhomloop}
There is an isomorphism of algeras
$$H_*(\Omega M) \cong T_\Z(s^{-1}(V_M^*))/(R_M^\perp).$$ 
\end{theorem}

Note that $V_M^* \cong \Z\{[\alpha_1],\ldots,[\alpha_r]\}$ and hence $s^{-1}(V_M^*) \cong \Z\{a_1,\ldots,a_r\}$. To prove the theorem recall from \cite{Ada56} that for a simply connected space $X$, the homology of $\Omega X$ is the homology of the cobar construction $\Omega C_*(X)$. Following the ideas of \cite{AdHil56}, we prove that for an $(n-1)$-connected $2n$-manifold $M$ the homology of $\Omega M$ may be computed using the cobar construction on $H_*(M)$. We will need the following result. 

\begin{prop}\label{mancoalg}
The coalgebra structure on  $H_*(M)\cong \Z\{1,[\alpha_i],[M] \}$, where $[M]$ stands for the orientation class, is given by formulae 
$$\Delta(1)=1\otimes 1,~ \Delta([\alpha_i])=[\alpha_i]\otimes 1 + 1\otimes [\alpha_i]$$
$$\Delta([M])=[M] \otimes 1 + 1\otimes [M] + \sum_{i,j} g_{i,j}\alpha_i\otimes \alpha_j. $$
The matrix $((g_{i,j}))$ is the matrix of the intersection form of $M$. 
\end{prop}

\begin{proof}
The coalgebra structure is dual to the ring structure on cohomology and hence the result follows. 
\end{proof}

Recall from \cite{AdHil56} the construction of a differential graded algebra associated to every CW complex $X$ with a unique $0$-cell and no $1$-cells, whose homology, as an associative algebra, is $H_*(\Omega X)$. Fix an enumeration of the cells of $X$ as $\{\lambda_i\}$ satisfying $\dim(\lambda_i)\le \dim(\lambda_{i+1})$.  Write the set $\{l_i\}$ so that $|l_i|=|\lambda_i|-1$, and define $A(X)$ to be the tensor algebra on $\{ l_i\}$.  

Next define a ring map $\psi_X: A(X)\to C_*(\Omega X)$. The characteristic map of the cell $\lambda_i$ by adjunction induces a map of an $|l_i|$-cell to $\Omega X$. Define $\psi_X(l_i)$ to be the image of the orientation class in $C_*(\Omega X)$. Now using the product structure extend this to $\psi_X: A(X)\to C_*(\Omega X)$. 

 The differential on $A(X)$ is defined inductively on the skeleta of $X$ on each generator. Suppose that the differential is defined on $A(X^{(m)})$ and $\psi_{X^{(m)}} : A(X^{(m)}) \to C_*(\Omega X^{(m)})$ is a quasi isomorphism. For an $(m+1)$-cell $\lambda_i$ consider $  \partial \lambda_i:S^m \to X^{(m)}$ and hence 
$$\hat{\lambda}_i :S^{m-1}\to \Omega S^m \stackrel{\partial \lambda_i}{\longrightarrow} \Omega X^{(m)}.$$
Choose $l\in A(X^{(m)})$ so that the homology class $\psi(l)$ represents the Hurewicz image of $\hat{\lambda}_i$. Define $d(l_i)=l$. Using Liebniz rule for graded differentials one obtains a differential graded algebra $A(X)$. From \cite{AdHil56} one knows that $A(X)$ is quasi-isomorphic to $C_*(\Omega X)$ as a differential graded algebra. 

\begin{prop}\label{mancobar}
The homology of the loop space $H_*(\Omega M)$, as an associative algebra, is the homology of the cobar construction on $H_*(M)$. 
\end{prop}

\begin{proof}
We need to compute $A(M)$ with respect to the CW complex structure 
$$M= \vee_r S^n \cup_\Lambda e^{2n}.$$
The generators of $A(M)$ are $a_i,\mu$ respectively for $[\alpha_i], [M]$ with $|a_i | = n-1$ and $ |\mu| = 2n-1$. The differentials for $a_i$ are $0$ as the attaching maps of the corresponding cells are trivial. 

It remains to compute $\partial \mu$. Suppose that $\partial \mu = l$. It follows that in the dimension range 
$0\leq * \leq 2n-2$, 
$$H_*(\Omega M) \cong T(a_1,\ldots,a_r)/(l).$$
We may compute $H_*(\Omega M)$ in this range using the Serre spectral sequence associated to the path-space fibration $\Omega M \to PM \to M$. This has the form 
$$E_2^{p,q} = H^p(M) \otimes H^q(\Omega M) \implies H^*(\mathit{pt}).$$
The homology of $M$ being torsion free implies that the cohomology of $M$ is just the dual. In the above range of dimensions $H_*(\Omega M)$ is also torsion free except possibly at $2n-2$. The first possible non-zero differential in the spectral sequence is $d_n$ and from the given identifications we have
$$d_n(a_i^*)=[\alpha_i]^*.$$
It follows that $d_n : E_n^{n,n-1} \to E_n^{2n,0}$ is given by 
$$d_n([\alpha_i]^* \otimes a_j^*) = g_{i,j} [M]^*$$
where $((g_{i,j}))$ is the matrix of the intersection form. Note that by these formulae $d_n$ is surjective onto the $(q=0)$-line so that 
$$d_n : E_2^{0,2n-2}\cong E_n^{0,2n-2} \to E_n^{n,n-1}$$
is injective onto the kernel of $d_n$. Hence $H^{2n-2}(\Omega M) \cong E_2^{0,2n-2}$ may be identified with the kernel of the map 
$$\Z\{a_i^*\}^{\otimes 2} \to \Z,\,\,a_i^*\otimes a_j^* \mapsto g_{i,j}.$$
From Poincar\'{e} duality one knows that the matrix $((g_{i,j}))$ is non-singular and therefore, $l(M)=\sum g_{i,j} a_i \otimes a_j$ is a primitive element in $\Z\{a_i\}^{\otimes 2}$. It follows that 
$$H_{2n-2}(\Omega M) \cong \Z\{a_i\}^{\otimes 2}/(l(M)).$$
Therefore $\partial \mu = \pm l(M)$. By changing the orientation on the top cell if necessary we may assume that the sign is $1$. Note that from Proposition \ref{mancoalg} this formula matches the differential in the cobar construction $\Omega H_*(M)$.
\end{proof}
\mbox{ }

\noindent{\it Proof of Theorem \ref{manhomloop}.}
Observe that the element $l(M)$ in the proof above generates $R_M^\perp$. By Proposition \ref{freemod}, $T_\Z(s^{-1}(V_M^*))/(l(M))$ is a free module. From the proof of Propostion \ref{mancobar} above the ring map 
$$T_\Z(a_1,\ldots,a_r)\longrightarrow H_*(\Omega M)$$
factors through
$$\phi:T_\Z(a_1,\ldots,a_r)/(l(M)) \longrightarrow H_*(\Omega M).$$
By Proposition \ref{mancobar}, the homology $H_*(\Omega M)$  can be computed as the homology of the cobar construction on $H_*(M)$. The coalgebra structure on $H_*(M)$ is computed in Proposition \ref{mancoalg}. 

Let $k=\Z/p\Z~\mbox{or}~\Q$, then $H_*(M;k)$ is the quadratic coalgebra $C(V_M, l(M))$ which by Proposition \ref{man-Kos} is Koszul. Therefore $H_*(\Omega M; k)$ is the Koszul dual algebra 
$$T_k(s^{-1}(V_M))/(l(M)) \cong T_\Z(s^{-1}(V_M^*))/(l(M)) \otimes k.$$
Hence the composite
$$ T_\Z(s^{-1}(V_M^*))/(l(M)) \otimes k \stackrel{\phi\otimes k}{\longrightarrow} H_*(\Omega M)\otimes k \to H_*(\Omega M;k)$$
is an isomorphism for $k= \Z/p\Z~\mbox{or}~\Q$. It follows that $Tor(k,H_*(\Omega M))=0$ and thus $H_*(\Omega M)$ is free. As $\phi$ is an isomorphism after tensoring with $\Z/p\Z$ and $\Q$, we obtain that  $\phi$ is an isomorphism over $\Z$.    
\qed

\subsection{From homology of the loop space to homotopy groups}

The element $l(M) = \sum_{i,j} g_{i,j} a_i\otimes a_j$ belongs to $T_\Z(a_1,\ldots,a_r)$. The matrix $((g_{i,j}))$ is symmetric if $n$ is even and skew-symmetric if $n$ is odd. This implies that the element $l(M)$ lies in free graded Lie algebra generated by $a_1,\ldots,a_r$ which we denote by  $\Lie^{gr}(a_1,\ldots,a_r)$.  Consider the graded Lie algebra $\LL_r^{gr}(M)$ (over $\Z$) given by 
$$\frac{\Lie^{gr}(a_1,\ldots,a_r)}{(l(M))}$$
where $(l(M))$ denotes the graded Lie algebra ideal generated by $l(M)$. 

We make an analogous ungraded construction. Consider the element 
$$l^u(M):= \sum_{i<j} g_{i,j} [a_i,a_j]\in \Lie(a_1,\ldots,a_r). $$
Note that $l^u(M)$ equals $l(M)$ if $n$ is odd. Now we denote by $\LL_r^u(M)$ the Lie algebra
$$\frac{\Lie(a_1,\ldots,a_r)}{(l^u(M))}.$$

Equip $\LL_r^u(M)$ with a grading defining the degree of $a_i$ to be $n-1$ (note that $l^u(M)$ is a homogeneous element of degree $2n-2$). Denote by $\LL_{r,w}^u(M)$ the degree $w$ homogeneous elements of $\LL_r^{u}(M)$. From Proposition \ref{freemod} and Theorem \ref{Liebasis} we know that $\LL^u_r(M)$ is a free module and the Lyndon basis gives a basis of $\LL^u_r(M)$.

Enlist the elements of the Lyndon basis in order as $l_1 < l_2 <\ldots$ and define the height of a basis element by $h_i= h(l_i)=r+1$ if $b(l_i) \in \LL_r^u(M)$. Then $h(l_i)\leq h(l_{i+1})$. Note that $b(l_i)$ represents an element of $\Lie(a_1,\ldots,a_r)$ and is thus represented by an iterated Lie bracket of $a_i$. Use the iterated Whitehead products to define maps $\lambda_i : S^{h_i}\rightarrow M$. 



\begin{theorem}\label{htpy}
There is an isomorphism\footnote{Note that the right hand side is a finite direct sum for each $\pi_n(M)$.}
$$\pi_*(M) \cong \sum_{i\geq 1} \pi_* S^{h_i}$$ 
and the inclusion of each summand is given by $\lambda_i$.
\end{theorem}

\begin{proof}
Write $S(t) = \prod_{i=1}^t \Omega S^{h_i}$. The maps $\Omega \lambda_i: \Omega S^{h_i} \rightarrow \Omega M$ for $i=1,\ldots,l$ can be multiplied using the $H$-space structure on $\Omega M$ to obtain a map from $S(l)\rightarrow \Omega T$. Use the model for $\Omega M$ given by the Moore loops so that the multiplication is strictly associative with a strict identity. Then the maps  $S(t) \rightarrow \Omega M$ and $S(t')\rightarrow \Omega M$ for $t\leq t'$ commute with the inclusion $S(t)\rightarrow S(t')$ by the basepoint on the last $t'-t$ factors.  Hence we obtain a map 
$$\Lambda: S := \mathit{hocolim}~ S(t) \rightarrow \Omega M$$   
We prove that $\Lambda$ is a weak equivalence by proving that $\Lambda_*$ is an isomorphism on integral homology. There are two cases when $n$ is even and $n$ is odd.

We know from Theorem \ref{manhomloop} that $H_*(\Omega M) \cong Ts^{-1}(V_M^*)/(l(M)) $. Note that 
$$s^{-1}(V_M^*) = \Z\{a_1,\ldots,a_r\}.$$
When $n$ is odd, $l(M)=l^u(M)$ which lies in the Lie algebra generated by $a_1,\ldots,a_r$.  Thus $Ts^{-1}(V_M^*))/(l(M))$ is the universal enveloping algebra of the Lie algebra 
$$\LL_r^u(M) \cong \Lie(a_1,\ldots a_r)/ (l^u(M)).$$
By the Poincar\'e-Birkhoff-Witt Theorem 
$$E_0H_*(\Omega M) \cong S(\LL_r^u(M))$$
where $E_0H_*(\Omega M)$ is the associated graded algebra of $H_*(\Omega M)$ with respect to the length filtration and $S(\LL_r^u)$ is the symmetric algebra on $\LL_r^u(M)$. 

The homology of $S$ is the algebra 
$$H_*(S) \cong T_\Z (c_{h_1 -1})\otimes T_\Z(c_{h_2-1})\otimes\cdots \cong \Z[c_{h_1-1},c_{h_2 -1},\ldots ]$$
where $c_{h_i -1}$ denotes the generator in $H_{h_i -1}(\Omega S^{h_i})$. Now each $c_{h_i-1}$ maps to the  Hurewicz image of $\Omega (\lambda_i)\in H_{h_i-1}(\Omega M)$.  Consider the composite 
$$\rho : \pi_n(X)\cong \pi_{n-1}(\Omega X) \stackrel{\mathit{Hur}}{\longrightarrow} H_{n-1}(\Omega X).$$
We know from \cite{Hil55}, Lemma 2.2 (also see \cite{Sam53}) that 
\begin{equation}\label{hur}
\rho([a,b])=\pm (\rho(a)\rho(b) - (-1)^{|a||b|}\rho(b)\rho(a)).
\end{equation}
 The map $\rho$ carries each $\alpha_i$  to $a_i$. The element  $b(l_i)$ is mapped inside $H_*(\Omega M)$ to the element corresponding to the graded Lie algebra element (upto sign) by equation (\ref{hur}). Denote the Lie algebra element (ungraded) corresponding to $b(l_i)$ by the same notation. We readily discover that the difference of the graded and ungraded elements lie in algebra generated by terms of lower weight, i.e.,
\begin{equation}
\label{rhobl}
\rho(b(l_i)) \equiv b(l_i)~~ (\!\!\!\!\mod~\mathit{lower~order~terms}).
\end{equation}
Therefore
$$E_0Ts^{-1}(V_M^*)/(l(M)) \cong \Z [b(l_1),b(l_2),\ldots] \cong  \Z [\rho(b(l_1)),\rho(b(l_2)),\ldots].$$
It follows that the monomials in $\rho(b(l_i))$, $i=1,2,\ldots$ form a basis of $Ts^{-1}(V_M^*)/(l(M))$ which is $H_*(\Omega M) $. The map $\Lambda : S\to \Omega M$ maps $c_{h_i-1} \to \rho(b_i)$ and takes the product of elements in $\Z[c_{h_1-1},c_{h_2 -1},\ldots ]$ to the corresponding Pontrjagin product. It follows that $\Lambda_*$ is an isomorphism. 


Now we consider the case that $n$ is even. Again it suffices to show that $\frac{T(a_1,\ldots,a_r)}{(l(M))}$ has a basis given by monomials on $\rho(b(l_1)),\rho(b(l_2)),\ldots$ where $\rho(b(l_i))$ is mapped as the above. First observe that $ T(a_1,\ldots,a_r)/(l(M))$ is universal enveloping algebra of the graded Lie algebra $\LL^{gr}_r(M)$ so that the Poincar\'e-Birkhoff-Witt Theorem (cf. Theorem \ref{PBW}) for graded Lie algebras implies 
$$E(\LL^{gr}_r(M)^{\mathit{odd}} ) \otimes P(\LL^{gr}_r(M)^{\mathit{even}} ) \cong  E_0 T(a_1,\ldots,a_r)/(l(M))$$
Next we show that monomials in $b(l_i)$ span $T(a_1,\ldots,a_r)/(l(M))$. In view of the isomorphism above it suffices to show that all the elements in $\LL^{gr}_r(M)$ can be expressed as linear combinations of monomials in $\rho(b(l_i))$. This is done inductively. It is clear for elements of weight $1$. For the weight $2$ elements note that they are generated by $[a_i,a_j]$ for $i<j,~ (i,j)\neq (1,2)$ and $a_i^2$. The former are the Lyndon words and the latter is the square of a monomial. 

Note that $\rho(l^u(M))= l(M)$ and $\rho$ carries ungraded Lie bracket to the graded Lie bracket. The elements of $\LL^{gr}_r(M)$ are linear combinations of graded Lie brackets and squares\footnote{From one of the conditions in the definition of a graded Lie algebra, the bracket with a square can be expressed as a bracket.}. The proper Lie brackets are the images of the corresponding ungraded brackets under $\rho$. The square is represented as the square of the corresponding monomial and hence can be expressed as a monomial in the $\rho(b(l_i))$. 

Hence we have that the graded map
$$ \Z[\rho(b(l_1)),\rho(b(l_2)),\ldots] \longrightarrow     T(a_1,\ldots,a_r)/(l(M))$$
is surjective. We also know 
$$                      \Z[b(l_1),b(l_2),\ldots]    \longrightarrow T(a_1,\ldots,a_r)/(l^u(M))$$
is an isomorphism. Now both $T(a_1,\ldots,a_r)/(l(M))$ and $T(a_1,\ldots,a_r)/(l^u(M))$ have bases given by the Diamond lemma and thus are of the same graded dimension. It follows that the graded pieces of $ \Z[\rho(b(l_1)),\rho(b(l_2)),\ldots]$ and $     T(a_1,\ldots,a_r)/(l(M))$ have the same finite rank. Thus on graded pieces one has a surjective map between free $\Z$-modules of the same rank which must be an isomorphism. Thus $\Lambda_*$ is an isomorphism. 

Therefore we have proved that $\Lambda_*$ is an isomorphism in either case. Both $S$ and $\Omega M$ are $H$-spaces, and hence simple (that is $\pi_1$ is abelian and acts trivially on $\pi_n$ for $n\ge 2$). It follows that $\Lambda$ is a weak equivalence. Hence 
$$\pi_*(M)\cong \pi_{*-1} (\Omega M)  \cong \pi_{*-1} S \cong \sum_{i\ge 1} \pi_{*-1} (\Omega S^{h_i}) \cong \sum_{i\ge 1} \pi_* S^{h_i}.$$
\end{proof}

Theorem \ref{htpy} can be used to find out the number of $\pi_s S^l$ which occurs in $\pi_s M$. 
\begin{theorem}\label{htpyform}
 The number of groups $\pi_s S^l$ in $\pi_s(M)$ is $0$ if $l$ is not of the form $d(n-1)+1$ and if $l=d(n-1)+1$ this number is 
$$\sum_{c|d}\frac{ \mu(c)}{c}\sum_{a+2b=\frac{d}{c}}(-1)^b{a+b \choose b}\frac{r^a}{a+b}.$$
\end{theorem}
\begin{proof}
It is enough to compute the dimension $l_d$ of the $d^\textup{th}$-graded part of the Lie algebra $\LL^u_r(M)$. We use the generating series to compute this from the universal enveloping algebra $H_*(\Omega M)$ as in \cite{BaBa15}.

The generating series for $H_*(\Omega M)$ is $p(t)= \frac{1}{1-rt^{n-1} + t^{2n-2}}$. From the Poincar\'e-Birkhoff-Witt theorem, the symmetric algebra on $\LL^u_r(M)$ is $H_*(\Omega M)$. Hence we have the equation 
$$\frac{1}{\prod_d (1-t^d)^{l_d}} =   \frac{1}{1-rt^{n-1} + t^{2n-2}}.$$
Take log of both sides:
\begin{eqnarray*}
\log (1-rt^{n-1}+t^{2n-2} ) & = & \sum_d l_d\log (1-t^d)\\
& = & -\sum_d l_d\left(t^d+\frac{t^{2d}}{2}+\frac{t^{3d}}{3}+\cdots\right).
\end{eqnarray*}
Expanding this and equating coefficients, we see that 
$$
\lambda_m:=\textup{coefficient of $t^m$ in $\log (1-rt^{n-1}+t^{2n-2})$} = -\frac{1}{m}\bigg(\sum_{d|m} d l_d \bigg).
$$
We use the M\"{o}bius inversion formula; it gives us 

$$l_m =-\sum_{d|m} \mu(d)\frac{\lambda_{m/d}}{d}.  $$
In fact, by expanding $\log (1-rt^{n-1}+t^{2n-2})$, we see that $\lambda_m =0$ if $m$ is not divisible by $n-1$ and 
$$\lambda_{d(n-1)}= -\sum_{a+2b=d}(-1)^b{a+b \choose b}\frac{r^a}{a+b}.$$
We conclude that $l_m$ is $0$ if $m$ is not divisible by $n-1$ and 
$$l_{d(n-1)} = \sum_{c|d}\frac{ \mu(c)}{c}\sum_{a+2b=\frac{d}{c}}(-1)^b{a+b \choose b}\frac{r^a}{a+b}.$$

\end{proof}

\subsection{The case when Betti number is 1}
\label{Betti1}
We consider the cases when $M$ is a closed $(n-1)$-connected $2n$-manifold with $H^n(M) = \Z$ so that $n\in \{2,4,8\}$. Curiously, unlike the cases when the $n^\textup{th}$ Betti number is at least $2$, in this case it is not possible to express the homotopy groups as a direct sum of the homotopy groups of some list of spheres. For example, the homotopy groups of the projective plane of Cayley Octanions is not of this form (cf. \cite{Mim67}). 

We carefully consider each of the cases $n=2,4,8$. The obvious examples of manifolds are $\C P^2$, $\Hy P^2,$ and $ \Oc P^2$. In dimension $4$ such a manifold is of the form $S^2 \cup_h e^4$, where $h\in \pi_3S^2\cong \Z$ is a Hopf invariant one class. As there is only one Hopf invariant one class in $\pi_3 S^2$, the only example is $\C P^2$. 

\begin{exam}
Let $M= \C P^2$. We may compute the homotopy groups using the fibration $S^1 \to S^5 \to \C P^2$. The map $S^1 \to S^5$ is null-homotopic. It follows that $\pi_k(\C P^2) \cong \pi_k S^5 \oplus \pi_{k-1}S^1 $ and hence 
$$\pi_k (\C P^2)=\left \{\begin{array}{rl} 
                \Z &\mbox{if}~k=2 \\
                \pi_k S^5  &\mbox{otherwise}.                  
\end{array}\right.$$   
\end{exam}

We shall also consider the cases $n=4$ and $n=8$. First compute the homology of the loop space. In each case $H^*(M) \cong \Z[x]/(x^3)$. Thus $H_*(M)\cong \Z\{1,\ep_n,\ep_{2n}\}$ has the coalgebra structure
$$\Delta (\ep_n)=\ep_n\otimes 1 + 1\otimes \ep_n,~ \Delta(\ep_{2n})= \ep_{2n}\otimes 1 + 1\otimes \ep_{2n} + \ep_n \otimes \ep_n.$$ 

\begin{prop}\label{bet1loophom}
Let $x_{n-1}$ denote the class in $H_{n-1}(\Omega M)$ obtained by looping the class $\ep_n$. Then $H_*(\Omega M)\cong \Z[z_{3n-2},x_{n-1}]/(x_{n-1}^2)$ for a class $z_{3n-2}$ in degree $3n-2$. 
\end{prop}

\begin{proof}
By Proposition \ref{mancobar}, as an associative algebra $H_*(\Omega M)$ is the homology of the cobar construction on $H_*(M)$. Thus, $H_*(\Omega M)$ is the homology of the differential graded algebra 
$$(T(x_{n-1}, x_{2n-1}),d),~ d(x_{n-1})=0,~d(x_{2n-1})=x_{n-1}\otimes x_{n-1}.$$  
Note that there is an ungraded isomorphism of this complex with the bar construction of the commutative algebra $\Z[x]/(x^3)$ with $x_{n-1}$ corresponding to the $s(x)$ and $x_{2n-1}$ to $s(x^2)$. The homology of the bar construction is $Tor_*^{\Z[x]/(x^3)}(\Z,\Z)$ which can be computed using an explicit resolution : 
$$\cdots \stackrel{\times x}{\longrightarrow}\frac{\Z[x]}{(x^3)}\stackrel{\times x^2}{\longrightarrow}\frac{\Z[x]} {(x^3)} \stackrel{\times x}{\longrightarrow}\frac{\Z[x]}{(x^3)} \to \Z.$$
It follows that $Tor_i^{\Z[x]/(x^3)}(\Z,\Z)\cong \Z$ for each $i$. We can match this up with the homology of the complex above. Let $z_{3n-2}= x_{2n-1} \otimes x_{n-1} + x_{n-1} \otimes x_{2n-1}$ and note that $z_{3n-2}^k$ is a cycle but not a boundary. The elements of $Tor_i$ are represented by the elements $z_{3n-2}^k$ for $i=2k$ and $z_{3n-2}^k x_{n-1}$ for $i=2k+1$. It follows that $H_*(\Omega M) \cong \Z[z_{3n-2},x_{n-1}]/(x_{n-1}^2)$. 
\end{proof}

Note  that the inclusion of the $n$-skeleton of $M$, which is homotopy equivalent to $S^n$, loops to give $\iota_x: S^{n-1} \to \Omega M$ whose Hurewicz image is the class $x_{n-1}$. Therefore the map $\iota_x$ is $(3n-3)$-connected and it follows that for $k\leq 3n-4$
$$\pi_{k+1} M \cong \pi_k \Omega M \cong \pi_k S^{n-1}.$$
In terms of the kind of formula we have (Theorem \ref{htpy}) a naive guess would be that $\pi_k M \cong \pi_{k-1} S^{n-1} \oplus \pi_k S^{3n-1}$. However, for that we would require a map $\beta : S^{3n-1} \to M$ such that $\rho(\beta) = z_{3n-2}$ using which one can construct a weak equivalence 
$$S^{n-1} \times \Omega S^{3n-1} \longrightarrow \Omega M.$$  
This is indeed the case when $n=2$. In the following we observe that when $n=4$ this is often the case but not always. In \cite{Mim67} it is shown that this does not hold even for the Cayley projective plane $\Oc P^2$. We shall prove that after inverting some primes the above formula holds. 

Recall that the attaching maps of the $2n$-cell lie in $\pi_7 S^4$ if $n=4$ and $\pi_{15} S^8$ if $n=8$. These groups are $\pi_7 S^4 \cong \Z\oplus \Z/12\Z$ and $\pi_{15} S^8 \cong \Z \oplus \Z/120\Z$. From methods in \cite{Tam61}, we may conclude the following proposition.
\begin{prop}
For every Hopf invariant one element $h:S^{2n-1} \to S^n$ the mapping cone of $h$ is homotopy equivalent to a topological manifold. 
\end{prop}

\begin{proof}
We may use the $J$-homomorphism $\pi_k SO(n) \to \pi_{n+k} S^n$. For $n=4,~k=3$, this gives a surjective map 
$$\Z \oplus \Z \cong \pi_3SO(4) \to \pi_7 S^4 \cong \Z \oplus \Z/12\Z$$
and for $n=8,~k=7$, this gives a surjective map 
$$\Z \oplus \Z \cong \pi_7 SO(8) \to \pi_{15} S^8 \cong \Z \oplus \Z/120\Z.$$
We may choose generators of $\pi_3SO(4)$ (respectively $\pi_7 SO(8)$) using conjugation and left multiplication in quarternions (respectively octanions). After applying $J$ the left multiplication gives the Hopf invariant one element and the image of the conjugation generates the other part. Thus all the Hopf invariant one elements are of the form $h_{m,1}= J(m,1)$ in this representation. 

Consider the element $(m,1) \in \pi_3SO(4)$ (respectively $\pi_7SO(8)$). This corresponds to a $4$-dimensional (respectively $8$-dimensional) bundle $\xi_{m,1}$ over $S^4$ (respectively $S^8$). The methods in \cite{Mil56} demonstrate that the associated sphere bundles $S(\xi_{m,n})$ to $\xi_{m,n}$ are homeomorphic to $S^7$ (respectively, $S^{15}$) if $n=1$ and the homeomorphism can be refined to a diffeomorphism if $m=0$. In these cases the composite 
$$S^7 \cong S(\xi_{m,1}) \to S^4$$
represents the element $h_{m,1} = J(m,1)$. Now use the disk bundle $D(\xi_{m,1})$ to form the union 
$$V_{m,1}= D(\xi_{m,1}) \cup_{(S(\xi_{m,1})\cong S^7)} D^8$$
to obtain a $8$-manifold homotopy equivalent to the mapping cone of $h_{m,1}$ (and analogously a $16$-manifold).  
\end{proof}

\begin{rmk}
Note that not all of the manifolds above can be smoothed. It is proved in \cite{Tam61} that $V_{m,1}^8$ is smooth implies $m(m+1) \equiv 0 ~(\mbox{mod~}4)$, whence $m=0,3,4,7,8,11\in \Z/12\Z$. Similarly, $V_{m,1}^{16}$ is smooth implies $m(m+1) \equiv 0~ (\mbox{mod~}8)$.
\end{rmk}

First we work with the case $n=4$. Note that the different Hopf invariant one elements in $\pi_7 S^4$ correspond to the different multiplications on $S^3$. These maps are obtained by applying the Hopf construction on the multiplication $S^3\times S^3\to S^3$ to get $S^7 = S^3 * S^3 \to S(S^3) = S^4$. It turns out that the computation of homotopy groups of $V_{m,1}$ is intricately related to the homotopy associativity of the multiplication which is discussed in \cite{Jam57}. We follow the notation in that paper. 

The left multiplication on $S^3$ induces the usual Hopf invariant one map $\gamma$ coming from quarternionic multiplication. This implies $h_{0,1} = \gamma$. The inverse of conjugation composed with left multiplication induces the opposite multiplication and hence $h_{-1,1}= \bar{\gamma}$. Thus we have $h_{m,1}=\gamma + m E\omega$.
We prove the following result.
\begin{theorem}\label{4assoc}
If $m\equiv 0~\mbox{or~}2 ~(\mbox{mod~} 3)$ then there is a weak homotopy equivalence 
$$\Omega V_{m,1}^8 \simeq S^3 \times \Omega S^{11}.$$ 
In particular, $\pi_k V_{m,1}^8 \cong \pi_{k-1} S^3 \oplus \pi_k S^{11}$. 
\end{theorem}

\begin{proof} Recall that if $X$ is a $H$-space then we have a quasi-fibration (cf. \cite{Sta70})
$$X \to X*X \to S(X)$$
where the last map is obtained by the Hopf construction. If $X$ is homotopy associative (that is, $X$ is an $A_3$-space) then using \cite{Sta63} we may construct a quasi-fibration 
$$X\to X*X*X \to M$$
where $M$ is the mapping cone of $X*X \to S(X)$. 

By \cite{Jam57}, the multiplication on $S^3$ corresponding to $h_{m,1}$ is homotopy associative for the assumed values of $m$. Thus, we have a quasi-fibration 
\begin{equation}
\label{Vm18}
S^3 \to S^{11} \to \mbox{~Cone}(h_{m,1}) (= V_{m,1}^8).
\end{equation}
Here $\textup{Cone}(h_{m,1})$ is the mapping cone of the map $S^3\ast S^3 \to S^4$ induced by $h_{m,1}$. 
The long exact sequence of  the fibration \eqref{Vm18} implies the last statement of the Theorem. 

Consider the inclusion of the $4$-skeleton $\iota:S^4\hookrightarrow V_{m,1}^8$. Since $S^4$ is the suspension of $S^3$, we may consider the associated map $f:S^3\to \Omega V_{m,1}^8$. Note that we have a map $\Omega V_{m,1}^8\to S^3$ using \eqref{Vm18} which when pre-composed with $f$ gives a self-map of $S^3$ which is an isomorphism on $\pi_{k\leq 3}$. By generalized Whitehead's Theorem, this implies that the self-map is a homotopy equivalence. Therefore $f$ is injective on homotopy groups. Moreover, this map is an isomorphism on $\pi_{k\leq 9}$. 

Let $e^{10}$ be the $10$-cell such that $S^3\cup e^{10}$ is the $10$-skeleton of  $\Omega V_{m,1}^8$. Then the attaching map of $e^{10}$ is trivial and we get a map $S^{10}\to \Omega V_{m,1}^8$. Using the multiplicative structure of $\Omega V_{m,1}^8$ we get a map $g:J(S^{10}) \to \Omega V_{m,1}^8$, where $J(X)$ is the James reduced product on $X$. Finally, combining the maps $f$ and $g$ we have a map $f\times g:S^3 \times J(S^{10}) \to \Omega V_{m,1}^8$. This is a homology isomorphism which we observe from the homology computation in Proposition \ref{bet1loophom}. Therefore
$$\Omega V_{m,1}^8 \simeq S^3 \times J(S^{10})$$
as both sides are simple spaces. Now using $J(S^n)\simeq \Omega S^{n+1}$, we arrive at the weak equivalence of the Theorem.
\end{proof}

Theorem \ref{4assoc} derives the desired expression of the homotopy groups of $ V_{m,1}^8$ when $m\neq 1,4,7,10\in \Z/12\Z$. It turns out that when $m$ takes these values the homotopy groups are different. In these cases we compute $\pi_{10}V_{m,1}^8$ and find it to be different from $\pi_9S^3$. From the inclusion $S^4 \to V_{m,1}^8$ we have the long exact sequence 
$$ \cdots \to\pi_k S^4 \to \pi_k V_{m,1}^8 \to \pi_k(V_{m,1}^8,S^4) \to \pi_{k-1} S^4\to \cdots $$  
The inclusion $\iota:S^4\hookrightarrow V_{m,1}^8$ is $7$-connected. Thus, the map $\iota_\ast:\pi_k S^4 \to \pi_k V_{m,1}^8$ is an isomorphism if $k\le 6$. From the excision property of homotopy groups\footnote{If a CW pair $(X,A)$ is $r$-connected and $A$ is $s$-connected then the map $\pi_i(X,A)\to \pi_i(X/A)$ induced by the quotient map $X\to X/A$ is an isomorphism for $i\leq r+s$.} we have $ \pi_k(V_{m,1}^8,S^4) \cong \pi_k S^8$ if $k\leq 10$. In these degrees the boundary map $\pi_k(V_{m,1}^8,S^4) \to \pi_{k-1} S^4$ is the composite 
$$\pi_k(V_{m,1}^8,S^4)\cong\pi_kS^8\cong \pi_{k-1}S^7 \stackrel{h_{m,1}}{\longrightarrow} \pi_{k-1} S^4$$
which is injective\footnote{The map $h_{m,1}:\pi_kS^7\to \pi_kS^4$ is injective for any $k$.}. Thus the map $\pi_{10} S^4 \to \pi_{10} V_{m,1}^8$ is surjective. 

Let $s\in \pi_8(V_{m,1}^8,S^4)$ be the characteristic map. From Theorem 1.4 of \cite{Jam54} we have 
$$\pi_{11}(V_{m,1}^8,S^4) \cong s_*\pi_{11}(D^8,S^7) \oplus \Z \{[\iota_4,s]\}.$$
The map $s_*\pi_{11}(D^8,S^7)\to \pi_{10}S^4$ is equivalent to the composite 
$$\pi_{11}(D^8,S^7) \cong \pi_{11}(S^8) \cong \pi_{10}S^7 \stackrel{h_{m,1}}{\longrightarrow} \pi_{10} S^4.$$
We use the notation in \cite{Jam57}. Recall $\pi_{10} S^7 \cong \Z_{24}\{\beta\}$ denoting $\beta$ by a generator. Let $\gamma$ denote the Hopf invariant one class from the Hopf construction of the quarternionic multiplication. Then, with $E$ denoting suspension, we have
$$\pi_{10}S^4 \cong \pi_{10}S^7 \oplus E(\pi_9 S^3) \cong \Z_{24}\{\gamma \circ \beta\} \oplus \Z_3 \{E\delta\}.$$
Let $E\omega$ denote a generator of $E(\pi_6 S^3)\subset \pi_7 S^4$ so that $h_{m,1}=\gamma + mE\omega$. As the class $\beta$ is in the image of $E$,
$$h_{m,1}\circ \beta = \gamma \circ \beta + mE\omega \circ \beta = \gamma \circ \beta - mE\delta.$$
The last equality is from \cite{Jam57} Equation (3.7b). The image $\partial [\iota_4, s]=[\iota_4,h_{m,1}]=[h_{m,1},\iota_4]$ can be computed as 
$$[h_{m,1},\iota_4] = (1+2m)\gamma\circ \beta + m E\delta.$$
Observe that if $m\equiv 1 ~(\mbox{mod}~3)$, 
$$\frac{ \Z_{24}\{\gamma \circ \beta\} \oplus \Z_3 \{E\delta\}}{\langle  \gamma \circ \beta - mE\delta,(1+2m)\gamma\circ \beta + m E\delta\rangle} = 0.$$
Thus for the above values of $m\equiv 1 ~(\mbox{mod}~3)$, $\pi_{10}V_{m,1}^8=0$ and so $\pi_{k-1} S^3$ is not a summand of $\pi_k V_{m,1}^8$ for $k=10$. However, we can prove an expression similar to Theorem \ref{4assoc} once the prime $3$ is inverted. 

\begin{theorem}\label{4bet1htpy}
For any $m$ 
$$\pi_k V_{m,1}^8 \otimes \Z[1/3] \cong (\pi_{k-1}S^3\otimes \Z[1/3])\oplus (\pi_k S^{11} \otimes \Z[1/3]). $$
If $m,m'$ are both congruent to $1 ~(\mbox{mod}~3)$ then $\pi_*V_{m,1}^8 \cong \pi_*V_{m',1}^8$. 
\end{theorem}

\begin{proof}
Since the obstructions to homotopy associativity of multiplications (on $S^3$) lie in $\pi_9S^3\cong \Z/3\Z$ (see \cite{Jam57}), once $3$ is inverted the techniques of \cite{Jam57} implies that all the multiplications on the $\Z[1/3]$ localised sphere $S^3_{1/3}$ are homotopy associative. We may now repeat the proof of Theorem \ref{4assoc} to deduce the first statement.  

Now turn to the second statement. We know that once $3$ is inverted the homotopy groups are isomorphic. It suffices to prove that the homotopy groups are isomorphic once $2$ is inverted. Recall that $V_{m,1}^8 \simeq \textup{Cone}(h_{m,1})$. We have maps 
$$\textup{Cone}(2h_{m,1}) \to \textup{Cone}(h_{m,1}),\,\,\textup{Cone}(4h_{m,1})\to \textup{Cone}(h_{m,1})$$
which are isomorphisms in $\Z[1/2]$-homology. For all $m\equiv 1 ~(\mbox{mod}~3)$, all the values of $4m$ are equal $(\mbox{mod}~12)$. Therefore there is a zigzag of $\Z[1/2]$-equivalences 
$$\textup{Cone}(h_{m,1})\leftarrow \textup{Cone}(4h_{m,1}) = \textup{Cone}(4h_{m',1}) \to \textup{Cone}(h_{m',1})$$   
which implies the result.
\end{proof}

Now consider $n=8$. We know that for the octonionic projective plane $V_{0,1}^{16}$ the homotopy groups are not split as a direct sum as above. However, we can deduce the following theorem.

\begin{theorem} \label{8bet1htpy}
For any $m$, 
$$\pi_k V_{m,1}^{16} \otimes \Z[1/6] \cong (\pi_{k-1}S^7\otimes \Z[1/6])\oplus (\pi_k S^{23} \otimes \Z[1/6]). $$
\end{theorem}

\begin{proof}
 We know that the Whitehead product $[\iota_7,\iota_7]=0$ so that $S^7$ is a $H$-space with $120$ different multiplications parametrized by $\pi_{14}S^7 \cong \Z/120\Z$. By \cite{Jam57}, the homotopy associativity obstruction lies in $\pi_{21}S^7$ and is non-zero for each of the multiplications. The different manifolds $V_{m,1}^{16}$ are homotopy equivalent to the mapping cones of the Hopf constructions on these multiplications.

Since $\pi_{21}S^7 \cong \Z/24\Z \oplus \Z/4\Z$, when $2$ and $3$ are inverted this group becomes $0$. Thus all the multiplications induced on the sphere localised at $6$ are homotopy associative. Now repeat the proof of Theorem \ref{4assoc} to deduce the Theorem.  
\end{proof}


\subsection{Computing rational homotopy groups}

As an application we may compute the rational homotopy groups of the $(n-1)$-connected $2n$-manifolds from above formulae. As in \cite{BaBa15}, we determine the precise dimensions of the rational homotopy groups and deduce exponential growth of geodesics when the manifolds are rationally hyperbolic. Note that the explicit formula for the dimension of rational homotopy groups were obtained in \cite{BerMad13} and explicit formula for the homology of the free loop space is obtained in \cite{BerBor15}. We reprove some of their results using the techniques above. 

\begin{theorem}\label{rathom}
Let $M^{2n}_{r}$ be a closed $(n-1)$-connected $2n$-manifold with $n^{th}$ Betti number $r\ge 2$. The rank $m_j(r)$  of $\pi_{j+1}(M^{2n}_{r})\otimes \Q$ is $0$ if $j$ is not a multiple of $n-1$ and when $j=d(n-1)$ we have
$$m_{d(n-1)} =(-1)^{d(n-1)}\sum_{c|d}(-1)^{\frac{d(n-1)}{c}}\frac{ \mu(c)}{c}\sum_{a+2b=\frac{d}{e}}(-1)^b{a+b \choose b}\frac{r^a}{a+b}.$$
\end{theorem}

\begin{proof}
 By Theorem \ref{manhomloop} we know that
$$H_*(\Omega M;\Q)\cong T(s^{-1}V_M^*)/(R_M^\perp).$$
Hence the generating series of the graded vector space $H_*(\Omega M;\Q)$ is 
$$\frac{1}{1-rt^{n-1}+t^{2n-2}}. $$
The rational homology $H_*(\Omega M;\Q)$ is isomorphic to the universal enveloping algebra of the rational homotopy Lie algebra $L:=\pi^\Q_*(\Omega M)$ (cf. \cite{MM65}). We may apply the Poincar\'{e}-Birkhoff-Witt theorem (Theorem \ref{PBW}) to deduce that $H_*(\Omega M ;\Q)$ is isomorphic to the free graded commutative algebra on $L$. The notation in the statement implies that the dimension of $L_j$ is $m_j(r)$. Hence we have the equation 
$$\frac{\prod_{i\,\textup{odd}} (1+t^i)^{m_i}}{ \prod_{i\,\textup{even}} (1-t^i)^{m_i}} =   \frac{1}{1-rt^{n-1} + t^{2n-2}}.$$
As in Theorem \ref{htpyform} we take log of both sides and obtain an equation from which we deduce the value of $m_j$ by M\"obius inversion formula. We have that $m_j$ is $0$ if $j$ is not a multiple of $(n-1)$ and  
$$m_{d(n-1)} =(-1)^{d(n-1)}\sum_{c|d}(-1)^{\frac{d(n-1)}{c}}\frac{ \mu(c)}{c}\sum_{a+2b=\frac{d}{c}}(-1)^b{a+b \choose b}\frac{r^a}{a+b}.$$
\end{proof}

Recall that simply connected, compact cell complexes either have finite dimensional rational homotopy groups or exponential growth of ranks of rational homotopy groups (cf. \cite{FHT01}, \S 33). The former are called rationally elliptic while the latter are called rationally hyperbolic.  The above formula imply the following result.

\begin{prop}
An $(n-1)$-connected $2n$-manifold is rationally elliptic if and only if the $n^{th}$ Betti number $r$ is at most $2$.
\end{prop}

\begin{proof}
As in \cite{BaBa15}, we observe that there are polynomials $p_j(x)$ with top coefficient $1/j$ such that the rank of $\pi_{j(n-1)+1}^\Q(M_r)$ is $p_j(r)$. If $r\geq 3$ then these polynomials satisfy 
$$\lim_{j\to \infty} (p_j(r))^{1/j}=\frac{r+\sqrt{r^2-4}}{2}.$$
Since this limit is not $0$, the rational homotopy groups cannot all vanish outside a finite stage. 

If $r\le 2$ then the manifolds are formal so their rational homotopy groups match one of $S^{2n},~J_2S^n,~S^n\times S^n,~J_2S^n\# (\pm J_2S^n)$. The space $J_2(S^n)$ is the second stage of the James construction having the CW complex structure $J_2S^n = S^n \cup_{[\iota,\iota]} e^{2n}$. For $\ep = \pm 1$, the space 
$$J_2S^n\# (\ep J_2S^n) = (S^n \vee S^n) \cup_{[\iota_1,\iota_1] + \ep [\iota_2,\iota_2]} e^{2n}.$$
In each of the cases the manifolds are rationally elliptic.    
\end{proof}

For rationally hyperbolic manifolds one is often interested in results about the growth of geodesics by estimating the free loop space homology. In this respect one may directly verify a conjecture of Gromov \cite{FOT08} for a class of manifolds.  

\begin{cor}
Let $M$ be an $(n-1)$-connected $2n$-manifold with $n^{th}$ Betti number at least $3$. For a generic metric on $M$, the number of geometrically distinct closed geodesics of length at most $\ell$ grows exponentially in $\ell$.
\end{cor}

\begin{proof}
We have seen that $M$ is rationally hyperbolic if the Betti number $r\geq 3$. All the $(n-1)$-connected $2n$-manifolds are rationally formal and coformal. Therefore from  \cite{Lam01b}, we conclude that the homology of the free loop space has exponential growth. It follows from \cite{BaZi82} that the number of geodesics grow exponentially for a generic metric.  
\end{proof}

As an application we can verify a conjecture of J. C. Moore (see introduction for details) for $(n-1)$-connected $2n$-manifolds. 
\begin{theorem}\label{Moorehcm}
The conjecture of Moore is valid for $(n-1)$-connected $2n$-manifolds with $n^\textit{th}$ Betti number at least $2$.
\end{theorem}   

\begin{proof}
We know that $(n-1)$-connected $2n$-manifolds are rationally elliptic if and only if the $n^\textup{th}$ Betti number is at most $2$. We gather from Theorem \ref{htpyform} that the homotopy groups are the same for all manifolds with the same Betti number. Therefore if the Betti number is $2$ then $\pi_*M \cong \pi_*(S^n \times S^n)$ and the latter has finite $p$-exponents for all primes. The case for $p=2$ is due to James \cite{Jam56} while the case for $p$ odd is due to Toda \cite{Tod56}; the precise exponent for odd primes (the $p$-exponent for $S^{2n+1}$ is $p^n$) is due to Cohen-Moore-Neisendorfer \cite{CMN79}.

It remains to verify that the exponents are not finite if the Betti number is at least $3$. In fact the $p$-exponents are not finite for any $p$. For this note that the Lie algebra $\LL^u(M)$ has infinite dimension if the Betti number of $M$ is at least $3$. It follows that in the decomposition in Theorem \ref{htpy} the sequence $h_i$ is unbounded. Hence for arbitrarily large $l$, $\pi_*S^l$ occurs as a summand of $\pi_*M$. For the conclution now observe that any $p^s$ may occur as the order of an element in $\pi_*S^l$ for arbitrarily large $l$. This follows from the fact that the same is true for the stable homotopy groups and the stable homotopy groups occur as unstable homotopy groups of $S^l$ for $l$ arbitrarily large. For example, torsion of order $p^s$ for any $s$ occurs in the image of the $J$-homomorphism (cf. \cite{Rav86}, Theorem 1.1.13).  
\end{proof}

\section{Applications and Examples}\label{Appex}
In \S \ref{conn}, we computed the homotopy groups of $(n-1)$-connected $2n$-manifolds in terms of homotopy groups of spheres. A special case is that of simply connected $4$-manifolds and, essentially, $4$-manifolds with finite fundamental group. The techniques of \S \ref{conn} can be applied to other cases as well - connected sums of products of spheres and CW complexes formed by attaching a $2n$-cell to a wedge of $n$-spheres. 

\subsection{Homotopy groups of simply connected 4-manifolds}\label{sc4d}
Simply connected $4$-manifolds are the first examples of the manifolds dealt with in the previous section in the case $n=2$. It is known that (see \cite{BaBa15}, \cite{DuLi05}) that the homotopy groups are determined by the second Betti number. By putting $n=2$ in Theorem \ref{htpyform}, Theorem \ref{rathom} and using \S \ref{Betti1} we have the following result.
\begin{theorem}\label{htpy4}
Let $M_r$ be a simply connected $4$-manifold with $H^2(M_r;\Z)\cong\Z^r$. \\
\textup{(i)} If $r=1$, then $\pi_2M_1= \Z$ and $\pi_s M_1 \cong \pi_s S^5$ if $s\geq 3$. \\
\textup{(ii)} If $r\geq 2$, then
$$\pi_s M_r \cong \sum_{k=2}^s (\pi_s S^k)^{g_k}, ~ g_k =\sum_{c|k-1} \frac{ \mu(c)}{c}\sum_{a+2b=\frac{k-1}{c}}(-1)^b{a+b \choose b}\frac{r^a}{a+b}. $$\\
\textup{(iii)} Suppose $m_d$ is the dimension of $\pi_{d+1}(M_r)\otimes \Q$ for $r\ge 2$. Then 
$$m_{d} =(-1)^{d}\sum_{c|d}(-1)^{\frac{d}{c}}\frac{ \mu(c)}{c}\sum_{a+2b=\frac{d}{c}}(-1)^b{a+b \choose b}\frac{r^a}{a+b}.$$
\end{theorem}

Fix the notation $M_r$ for a simply connected $4$-manifold with $2^\textup{nd}$ Betti number $r$. We can write the homotopy groups of a $4$-manifolds with finite fundamental group in terms of homotopy groups of $M_r$ and hence in terms of the homotopy groups of spheres. Observe that the formula of such homotopy groups depends only on the order of the fundamental group.

\begin{cor}
Suppose $M$ is a connected, closed $4$-manifold with $|\pi_1 M|= l$ which is finite. Let $r$ be the $2^{nd}$ Betti number of $M$. For $k\geq 2$, $\pi_k(M) \cong \pi_k(M_{l(r+2)-2})$. 
\end{cor}

\begin{proof}
Recall that for $k\geq 2$, $\pi_k M \cong \pi_k N$ where $N$ is a covering space of $M$. Choose $N$ to be the universal cover of $M$. From Theorem \ref{htpy4}, $\pi_k N \cong \pi_k M_t$ where $t$ is the second Betti number of $N$. It remains to compute $t$. We know that $\chi(N) = l. \chi(M)$. Hence $2+t= l(2+r)$. Therefore $t=l(r+2)-2$ and the result follows. 
\end{proof}

\subsection{Homotopy groups of connected sums of sphere products}\label{connsppr}
We consider manifolds which are connected sums of products of spheres where each constituent sphere is simply connected. Such manifolds are obtained from a wedge of spheres by adding a cell along some combination of Whitehead products. We write down the homotopy groups of these manifolds in terms of the homotopy groups of spheres. It is interesting to note that the final expression does not depend on the choice of orientation used to form the connected sum. 

\begin{rmk}
The homotopy groups of a connected sum of spheres is related to the computation of homotopy groups of simply connected $4$-manifolds. Recall from \cite{BaBa15}, \cite{DuLi05} that for a closed, simply connected $4$-manifold $M_r$ with Betti number $r\ge 2$ there is a fibre bundle 
$$S^1 \to \#^{r-1}(S^2\times S^3) \to M_r.$$
Therefore for $k\ge 3$, $\pi_k(M_r) \cong \pi_k(\#^{r-1}(S^2\times S^3))$.    
\end{rmk}

Fix a dimension $n$. We define, following \cite{BeTh14}, a connected sum of sphere products to be an $n$-manifold $T$ of the form 
 $$T= (S^{p_1}\times S^{n-p_1})\# \ldots \#(S^{p_r}\times S^{n-p_r})$$ 
where the connected sum is taken using some choice of orientation and $p_i,n-p_i\geq 2$. Since each constituent sphere is simply connected, this necessarily implies that $n\geq 4$. We want to compute $\pi_j T$ in terms of homotopy groups of spheres using the methods described in earlier sections. 


We start by fixing some notation. For $T$ as above denote $q_i = n - p_i$ and $\alpha_i, \beta_i$ ($i=1,2,\ldots,r$) the composites 
$$\alpha_i : S^{p_i}\to S^{p_i}\vee S^{q_i} \simeq S^{p_i}\times S^{q_i} - \{\mathit{pt}\}\to T$$
$$\beta_i : S^{q_i}\to S^{p_i}\vee S^{q_i} \simeq S^{p_i}\times S^{q_i} - \{\mathit{pt}\} \to T.$$
 The manifold $T$ has a CW complex structure with $2r$ cells of dimension  $p_1,\ldots,p_r$, $q_1,\ldots, q_r$, one $0$-cell and one $n$-cell. These $2r$ cells are attached to the $0$-cell by the constant map. The maps of these cells into $T$ are given by $\alpha_1,\ldots,\alpha_r$,$\beta_1,\ldots,\beta_r$; we use the same letters to denote the cells. The $n$-cell is attached along the sum of Whitehead products given by 
$$\gamma=\ep_1[\alpha_1,\beta_1]+ \ldots +\ep_r [\alpha_r,\beta_r]$$
where $\ep_i=\pm 1$ are signs according to the choice of orientation made while forming the connected sum $T$. 

Our first task is to compute the homology of the loop space $\Omega T$. We denote by $a_i,b_i$ the Hurewicz images of $\alpha_i,\beta_i$ respectively in $H_*(\Omega T;\Z)$. Then using the Pontrjagin ring structure on $H_*(\Omega T)$ we get a ring map 
$$\phi: T_\Z(a_1,\ldots,a_r,b_1,\ldots,b_r)\rightarrow H_\ast(\Omega T).$$
The Hurewicz homomorphism 
$$H : \pi_s(T) \cong \pi_{s-1}(\Omega T) \to H_{s-1}(\Omega T)$$
carries the Whitehead product to the graded Lie bracket (up to a sign; cf. \cite{Sam53}) on $H_*(\Omega T)$ defined by 
$$[\lambda,\mu] = \lambda. \mu - (-1)^{|\lambda||\mu|} \mu . \lambda$$
using the Pontrjagin product on $H_*(\Omega T)$. Thus the class $\gamma\in \pi_{n-1}(\vee_i (S^{p_i}\vee S^{q_i}))$ maps to $y\in H_{n-1}(\Omega ( \vee_i( S^{p_i}\vee S^{q_i})))$ given by 
$$y=\pm( \ep_1 [a_1,b_1] + \ldots + \ep_r [a_r,b_r]).$$
Since $\phi(y)=0$ there is an induced map
$$\phi: T_\Z(a_1,\ldots,a_r,b_1,\ldots,b_r)/(y)\to H_*(\Omega T)$$
where $(y)$ denotes the $2$-sided ideal generated by $y$.

\begin{theorem}\label{homloop}
The map 
$$\phi:T_\Z(a_1,\ldots,a_r,b_1,\ldots,b_r)/(y)\to H_*(\Omega T) $$ 
defined above is an isomorphism. 
\end{theorem}

We again compute  $H_*(\Omega T)$ by computing the homology of the cobar construction $\Omega C_*(T)$.  We will need the following result.
\begin{prop}\label{coalg}
Let  $\zeta$ denote the orientation class of $T$. The coalgebra structure on $H_*(T)$ is given by 
$$\Delta(1)=1\otimes 1,~ \Delta(\alpha_i)=\alpha_i\otimes 1 + 1\otimes \alpha_i,~ \Delta(\beta_i)= \beta_i\otimes 1 + 1\otimes \beta_i$$
$$\Delta(\zeta)=\zeta \otimes 1 + 1\otimes \zeta + \sum_{i=1}^r \ep_i(\alpha_i\otimes \beta_i +(-1)^{p_i+q_i} \beta_i \otimes \alpha_i).$$
\end{prop}





As in Proposition \ref{mancobar} we may verify that $H_*(\Omega T)$ is the homology of the cobar construction on $H_*(T)$ as an associative algebra. Using this we prove Theorem \ref{homloop}.


\mbox{ }

{\it Proof of Theorem \ref{homloop}.}
First observe that both the domain and codomain of $\phi$ are free $\Z$-modules. The domain is free by Proposition \ref{freemod} as $y$ is an anti-symmetric element. In order to check that $H_*(\Omega T)$ is free it suffices to prove that 
$$Tor(H_*(\Omega T),\Z/p\Z)=0$$
for every index $*$ and prime $p$. 

Let $k=\Z/p\Z~\mbox{or}~\Q$ and let $V$ be the vector subspace of  $H_*(T;k)$ spanned by $\alpha_i,\beta_i$. Let $\hat{\gamma}$ be defined by 
$$\hat{\gamma} := \sum_{i=1}^r \ep_i(\alpha_i\otimes \beta_i +(-1)^{p_i +q_i} \beta_i \otimes \alpha_i).$$
 From Proposition \ref{coalg} it is clear that $H_*(T;k)$ is the quadratic coalgebra $C(V,\hat{\gamma})$. 

Following the notation of Proposition \ref{Kos2} we introduce the order 
$$\alpha_1<\beta_1<\ldots <\alpha_r<\beta_r.$$
Observe that the expression $\hat{\gamma}$ is of the form ($\pm \beta_r\alpha_r+\mbox{lower terms}$) and hence the quadratic coalgebra is Koszul. It follows that the homology of the cobar construction is the Koszul dual algebra 
$$A(s^{-1}V,s^{-2}\hat{\gamma})\cong T(a_1,\ldots ,a_r,b_1,\ldots ,b_r)/(y)\cong T_\Z(a_1,\ldots ,a_r,b_1,\ldots ,b_r)/(y)\otimes k$$   
where the last isomorphism comes from Proposition \ref{irr}. Therefore the composite 
$$A_k(V,R)\cong A(V,R)\otimes k \stackrel{\phi\otimes k}{\longrightarrow} H_*(\Omega T)\otimes k \to H_*(\Omega T;k)$$
is an isomorphism. It follows that $H_*(\Omega T) \otimes k\to H_*(\Omega T;k)$ is surjective so that the groups $Tor(H_*(\Omega T),k)$ are $0$ and $H_*(\Omega T;k)\cong H_*(\Omega T)\otimes k$. Therefore $H_*(\Omega T)$ is a free $\Z$-module.

Thus both sides of $\phi$ are graded $\Z$-modules which are finitely generated and free in each degree. Hence it suffices to check that\\
(i) $\phi \otimes \Z/p\Z$ is an isomorphism for each prime $p$, and \\
(ii) $\phi \otimes \Q$ is an isomorphism. \\
These have been checked above and the result follows. \qed

\mbox{ }

The next step involves passing from the above computation of homology of the loop space to a computation of homotopy groups as in the previous section. Consider the graded Lie algebra $\LL_r^{gr}$ (over $\Z$) given by 
$$\frac{\Lie^{gr}(a_1,\ldots,a_r,b_1,\ldots,b_r)}{(y)}$$
and an analogous ungraded construction
 $$\LL_r^u= \frac{\Lie(a_1,\ldots,a_r,b_1,\ldots,b_r)}{(y^{u})}$$
where $y^u= \ep_1 [a_1,b_1]^u + \ldots + \ep_r [a_r,b_r]^u$ with the ungraded brcket $[x,y]^u=x.y - y.x$.  

Equip $\LL_r^u$ with a grading defining $|a_i|=p_i-1$, $|b_i|=q_i-1$. Denote by $\LL_{r,w}^u$ the degree $w$ homogeneous elements of $\LL_r^{u}$. From Proposition \ref{freemod} and Theorem \ref{Liebasis} we know that $\LL^u_r$ is a free module and the Lyndon basis gives a homogeneous basis of $\LL^u_r$.

Enlist the elements of the Lyndon basis in order as $l_1 < l_2 <\ldots$ so that $|b(l_i)|\leq |b(l_{i+1})|$. Define the height of a basis element by $h_i= h(l_i)=|b(l_i)|+1$. Note that $b(l_i)$ is represented by an iterated Lie bracket of $a_i$'s and $b_i$'s. Use the corresponding iterated Whitehead product of $\alpha_i,\beta_j$ to define maps $\lambda_i : S^{h_i}\rightarrow T$. 



\begin{theorem}\label{htpyT}
There is an isomorphism
$$\pi_* T \cong \sum_{i\geq 1} \pi_* S^{h_i}$$ 
and the inclusion of each summand of the right hand side\footnote{Note that the right hand side is a finite direct sum for each $n$.} is induced by $\lambda_i$. 
\end{theorem}

\begin{proof}
Write $S(t) = \prod_{i=1}^t \Omega S^{h_i}$. As in Theorem \ref{htpy}, we obtain a map 
$$\Lambda: S := \mathit{hocolim}~ S(t) \rightarrow \Omega T$$   
We prove that $\Lambda_*$ is an isomorphism. Since both $S$ and $\Omega T$ are $H$-spaces, it follows that $\Lambda$ is a weak equivalence. 

We know from Theorem \ref{homloop} that $H_*(\Omega T) \cong T(a_1,\ldots,a_r,b_1,\ldots,b_r)/(y)$ which is the universal enveloping algebra of the Lie algebra $\LL_r^u$. Filter $H_*(\Omega T)$ by the monomial degree. By the Poincar\'e-Birkhoff-Witt Theorem, the associated graded algebra of $H_*(\Omega T)$ is isomorphic to the symmetric algebra on $\LL_r^u$. 

The homology of $S$ is the polynomial algebra 
$$H_*(S) \cong \Z[c_{h_1-1},c_{h_2 -1},\ldots ]$$
where $c_{h_i -1}$ denotes the generator in $H_{h_i -1}(\Omega S^{h_i})$. Each $c_{h_i-1}$ maps to the  Hurewicz image of $\Omega (\lambda_i)\in H_{h_i-1}(\Omega T)$. As in Theorem \ref{htpy} we use $\rho$ for this map. Similar arguments from the proof of Theorem \ref{htpy} apply here. The equation \eqref{hur} holds, $\rho$ carries each $\alpha_i$  to $a_i$ and the difference of the graded and ungraded elements are generated by terms of lower weight. Hence \eqref{rhobl} holds and we have isomorphisms
$$E_0(T(a_1,\ldots,a_r,b_1,\ldots,b_r)/(y)) \cong \Z [b(l_1),b(l_2),\ldots] \cong  \Z [\rho(b(l_1)),\rho(b(l_2)),\ldots].$$
The map $\Lambda : S\to \Omega T$ maps $c_{h_i-1} \to \rho(b_i)$ and takes the product of elements in $\Z[c_{h_1-1},c_{h_2 -1},\ldots ]$ to the corresponding Pontrjagin product. It follows that $\Lambda_*$ is an isomorphism. 
\end{proof}

We may compute a formula for the number of $\pi_s S^m$ that appears in $\pi_s T$ as in Theorem \ref{htpyform}. Let 
$$
\eta_m:=\textup{coefficient of $t^m$ in $\log \big(1-\sum_{i=1}^r (t^{p_i-1}+t^{q_i-1}) + t^{n-2}\big)$}.$$

\begin{theorem}\label{htpyTform}
The number of groups $\pi_s S^j$ in $\pi_s T$ is given by
$$l_{j-1} =-\sum_{d|j-1} \mu(d)\frac{\eta_{(j-1)/d}}{d}.  $$
\end{theorem}
\begin{proof}
It is enough to compute the dimension $l_d$ of the $d^\textup{th}$-graded part of the Lie algebra $\LL^u_r$. We use the generating series to compute this from the universal enveloping algebra $H_*(\Omega T)$ as in \cite{BaBa15}.

The generating series for $H_*(\Omega T)$ is 
$$p(t)= \frac{1}{1-\sum_{i=1}^r (t^{p_i-1}+t^{q_i-1}) + t^{n-2}}.$$
From the Poincar\'e-Birkhoff-Witt theorem the symmetric algebra on $\LL^u_k(M)$ is $H_*(\Omega M)$. Hence we have the equation 
$$\frac{1}{\prod_d (1-t^d)^{l_d}} =   \frac{1}{1-\sum_{i=1}^r (t^{p_i-1}+t^{q_i-1}) + t^{n-2}}$$
Take log of both sides :
\begin{eqnarray*}
\log \Big({1-\sum_{i=1}^r (t^{p_i-1}+t^{q_i-1}) + t^{n-2}}\Big) & = & \sum_d l_d\log (1-t^d)\\
& = & -\sum_d l_d\left(t^d+\frac{t^{2d}}{2}+\frac{t^{3d}}{3}+\cdots\right)
\end{eqnarray*}
Expanding this and equating coefficients, we see that 
$$
\eta_m= -\frac{1}{m}\Big(\sum_{d|m} d l_d \Big).
$$
We use the M\"{o}bius inversion formula; it gives us 

$$l_m =-\sum_{d|m} \mu(d)\frac{\eta_{m/d}}{d}.  $$
\end{proof}
\noindent As an application we can verify a conjecture of J. C. Moore for connected sum of sphere products. 
\begin{theorem}\label{MooreT}
The conjecture of Moore is valid for connected sum of sphere products, i.e., if $T$ is a connected sum of sphere products then $\pi_\ast T$ has finite $p$-exponents for every prime $p$ if and only if $T$ is a sphere product if and only if $T$ is rationally elliptic.
\end{theorem}   

\begin{proof}
Let $T$ be a connected sum of $k$ sphere products:
$$T:=S^{p_1}\times S^{n-p_1}\#\cdots\#S^{p_r}\times S^{n-p_r}.$$
We know that if $r=1$ then $T$ is rationally elliptic and 
$$\pi_\ast T=\pi_\ast(S^{p_1})\times\pi_\ast(S^{n-{p_1}})$$
has finite $p$-exponents for any prime $p$ (cf. proof of Theorem \ref{Moorehcm}). 

Now consider $r>1$. It follows from \cite{FHT01} Theorem 33.3 that if $\pi_j(T)\otimes\mathbb{Q}\neq 0$ for some $j>2n$ then $T$ is rationally hyperbolic. Recall from the discussion before Theorem \ref{htpyT} about the Lie model for $T$. From our construction of a basis of free Lie algebra via Lyndon words, if $|a_1|\geq|b_1|, |a_2|\geq |b_2|$ we observe that the Lyndon word
$$a_1b_1a_2b_2a_1a_2$$
corresponds to a non-zero element, of degree $3n-2>2n$, in $\mathcal{L}^u_r$.  

Moreover, the $p$-exponents are not finite for any $p$. For this note that the Lie algebra $\LL^u_k(M)$ has infinite dimension if $r\geq 2$. It follows that in the decomposition in Theorem \ref{htpy} the sequence $h_i$ is unbounded. Hence for arbitrarily large $l$, $\pi_*S^l$ occurs (cf. proof of Theorem \ref{Moorehcm}) as a summand of $\pi_*M$. 
\end{proof}

As an upshot, we verify a conjecture of Gromov.
\begin{cor}
For $T$, a connected sum of sphere products, the number of distinct closed geodesics associated to a generic metric grows at least exponentially with respect to length if $T$ is rationally hyperbolic.
\end{cor}
\begin{proof}
We have seen that $r\geq 2$ are precisely the rationally hyperbolic ones among all possible cases of $T$. It follows from Lambrechts result  \cite{Lam01} that the ranks of the cohomology groups of the free loop space $L(M_1\#M_2)$ of a connected sum $M_1\# M_2$ grows at least exponentially if the rings $H^\ast(M_i)$ are not monogenic for $i=1,2$. This implies that the ranks of $H^\ast(LT)$ grow at least exponentially fast. Combining this with Ballmann-Ziller's result \cite{BaZi82}, we deduce that $T$ has the property for closed geodesics as claimed. 
\end{proof}

\subsection{Homotopy groups of some CW complexes}\label{hgCW}

We widen our view from $(n-1)$-connected $2n$-manifolds to CW complexes of a similar type. In this case the cup product pairing is no longer non-singular. Making appropriate assumptions on rank we can write down a formula for the homotopy groups of such a complex after inverting a finite set of primes. 

Suppose that $X$ is a CW complex with $r$ $n$-cells and one $2n$-cell and $n\geq 2$, i.e.,
$$X= \vee_r S^n \cup_f e^{2n}$$
where $f\in \pi_{2n-1} (\vee_r S^n)$. Let $\alpha_1,\ldots, \alpha_r$ denote the $r$ $n$-cells and also denote the maps $S^n\to T$. Fix a generator $\zeta \in H_{2n}(X)\cong \Z$. Let $\alpha_i^*$ denote the dual basis of $H^n(X)\cong \Z^r$. We have a bilinear form $Q: H_n(X) \times H_n (X) \to \Z$ given by
$$Q(\alpha_i,\alpha_j) = \langle \alpha_i^* \cup \alpha_j^* , \zeta \rangle $$
The form $Q$ is symmetric if $n$ is even and skew-symmetric if $n$ is odd. We have the following result.
\begin{prop}\label{rank2quad}
Suppose that $\mathit{Rank}(Q\otimes \Q) \geq 2$. Then $H^*(X)\otimes \Q$ is a quadratic algebra and there is a finite set of primes $\Pi_Q = \{p_1,\ldots, p_l\}$ such that for $p\notin \Pi_Q$, $H^*(X) \otimes \Z/p\Z$ is a quadratic algebra.  
\end{prop}   

\begin{proof}
Recall the proof of Proposition \ref{quadratic} and readily deduce that $H^*(X) \otimes k$ is a quadratic algebra if $\mathit{Rank}(Q\otimes k) \geq 2$. Hence the assertion for $\Q$ is clear. We have to prove that away from a finite set of primes $Q\otimes \Z/p\Z$ has rank at least $2$. 

Denote also by $Q$ the matrix of the form $Q$. This is symmetric if $n$ is even and skew-symmetric if $n$ is odd. In this case $\mathit{Rank}(Q\otimes \Q) \geq 2$ implies that the matrix has rank $\geq 2$ over the field $\Q$. That is, there is some $2\times 2$ minor with determinant $d \neq 0$. Thus if $p$ does not divide $d$ then $Q \otimes \Z/p\Z$ has rank at least $2$.   
\end{proof}

Fix the set of primes $\Pi_Q$ as above and let 
$$R_Q:= \Z\big[{\textstyle \frac{1}{p}}\,|\,p\in \Pi_Q\big].$$
We will compute $\pi_*(X) \otimes R_Q$ using our methods. First we need to compute $H_*(\Omega X; R_Q)$. Define, as before, $a_i\in H_{n-1}(\Omega X)$ to be obtained by looping the classes $\alpha_i$. The coalgebra structure is given as follows.

\begin{prop}\label{cellcoalg}
The coalgebra structure on $H_*(X)\cong \Z\{1,[\alpha_i],\zeta \}$, where $\zeta$ stands for a fixed generator of $H_{2n}(X)$, is given by formulae 
$$\Delta(1)=1\otimes 1,~ \Delta([\alpha_i])=[\alpha_i]\otimes 1 + 1\otimes [\alpha_i]$$
$$\Delta(\zeta)=\zeta \otimes 1 + 1\otimes \zeta + \sum_{i,j} \ep_{i,j}\alpha_i\otimes \alpha_j $$
where the matrix $((\ep_{i,j}))$ is the matrix of the form $Q$. 
\end{prop}

\begin{proof}
Note that $H_*(X)$ is torsion free so that $H^*(X)$ is the linear dual of $H_*(X)$. The cup product in $H^*(X)$ and the coalgebra structure on $H_*(X)$ are dual. The multiplication formula on $H^*(X)$ is given by $\alpha_i^* \alpha_j^* = \ep_{i,j} \zeta^*$. This gives the above formulae on the coalgebra structure. 
\end{proof}

Define $z$ to be the class obtained by looping the class $\zeta$. We have a statement similar to Proposition \ref{mancobar}.
\begin{prop}\label{cellcobar}
The homology $H_*(\Omega X; R_Q)$ is the homology of the cobar of the coassociative coalgebra $H_*(X;R_Q)$.
\end{prop}

\begin{proof}
As in the proof of Proposition \ref{mancobar} we have $H_*(\Omega X)$ is the homology of the complex $A(X)$. Therefore $H_*(\Omega X;R_Q)$ is the homology of the complex $A(X)\otimes R_Q$.

In $A(X)= (T(a_1,\ldots,a_r,z),d)$ observe that $d(a_i)=0$ as the attaching maps of the corresponding cells are trivial. It remains to compute $dz$. Suppose that $dz = l$. It follows that in the dimension range $0\leq * \leq 2n-2$, $H_*(\Omega X) \cong T(a_i)/(l)$. We may compute $H_*(\Omega X)$ in this range using the Serre spectral sequence associated to the fibration $\Omega X \to PX \to X$. This has the form 
$$E_2^{p,q} = H^p(X) \otimes H^q(\Omega X) \implies H^*(\mathit{pt}).$$
In the above range of dimensions $H_*(\Omega X)$ is also torsion free except possibly at $2n-2$. The first possible non-zero differential in the spectral sequence is $d_n$ and from the given identifications we have
$$d_n(a_i^*)=[\alpha_i]^*$$
It follows that $d_n : E_n^{n,n-1} \to E_n^{2n,0}$ is given by 
$$d_n([\alpha_i]^* \otimes a_j^*) = \ep_{i,j} \zeta^*$$
Note that by these formulae $d_n$ is surjective onto the $q=0$-line so that 
$$d_n : E_2^{0,2n-2}\cong E_n^{0,2n-2} \to E_n^{n,n-1}$$
is injective onto the kernel of $d_n$. Hence $H^{2n-2}(\Omega X) \cong E_2^{0,2n-2}$ may be identified with the kernel of the map 
$$\Z\{a_i^*\}^{\otimes 2} \to \Z,\,\,a_i^*\otimes a_j^* \mapsto \ep_{i,j}.$$
As $((\ep_{i,j}))$ has rank $\geq 2$ over $R_Q$  we have $l=\sum \ep_{i,j} a_i \otimes a_j$ is primitive in $R_Q\{a_i\}^{\otimes 2}$. It follows that 
$$H_{2n-2}(\Omega X;R_Q) \cong R_Q\{a_i\}^{\otimes 2}/(l)$$
and hence $dz = \pm l$. By changing the orientation on the top cell if necessary we may assume that the sign is $1$. 
\end{proof}

We may compute the required cobar construction in a similar way to obtain the following theorem.

\begin{theorem}\label{cellloop}
There is an isomorphism
$$H_*(\Omega X; R_Q) \cong T_{R_Q}(a_1,\ldots,a_r)/(l)$$ 
where $l= \sum \ep_{i,j} a_i a_j$.  
\end{theorem}

\begin{proof}
Observe that the element $l$ in the proof above generates $R^\perp$. By Proposition \ref{freemod} the $R_Q$-module $ T_{R_Q}(a_1,\ldots,a_r)/(l)$ is free. It follows from the proof above that the ring map 
$$T_{R_Q}(a_1,\ldots,a_r)\rightarrow H_*(\Omega X)$$
factors through
$$T_{R_Q}(a_1,\ldots,a_r)/(l) \rightarrow H_*(\Omega X).$$
Let us denote this map by $\phi$. From Proposition \ref{cellcobar} the homology $H_*(\Omega X)$ can be computed as the homology of the cobar construction on $H_*(X)$. The coalgebra structure on $H_*(X)$ is computed in Proposition \ref{cellcoalg}. 

Let $k=\Z/p\Z~\mbox{for}~p\notin \Pi_Q~\mbox{or}~\Q$, then $H_*(X;k)$ is a quadratic coalgebra which by Proposition \ref{man-Kos} is Koszul. Hence $H_*(\Omega X; k)$ is the Koszul dual algebra $T_k(a_1,\ldots,a_r)/(l)$ which is isomorphic to $T_{R_Q}(a_1,\ldots,a_r)/(l) \otimes_{R_Q} k$. Hence the composite
$$ T_{R_Q}(a_1,\ldots,a_r)/(l) \otimes_{R_Q} k \stackrel{\phi\otimes k}{\longrightarrow} H_*(\Omega X;R_Q)\otimes_{R_Q} k \to H_*(\Omega X;k)$$
is an isomorphism for $k= \Z/p\Z~\mbox{for}~p\notin \Pi_Q~\mbox{or}~\Q$. It follows that $Tor(k,H_*(\Omega X))=0$ for $k = R_Q/\pi$ where $\pi$ is a prime in $R_Q$. Thus $H_*(\Omega X;R_Q)$ is free. As $\phi$ is an isomorphism after tensoring with $R_Q/\pi$ and $\Q$, we obtain that  $\phi$ is an isomorphism over $R_Q$.    
\end{proof}

Next we use the expression for the homology of the loop space to compute the homotopy groups. Since we have the formula after inverting the finite set of primes $\Pi_Q$, the final result is obtained in the category of spaces localised at $R_Q$. 

The element $l = \sum_{i,j} \ep_{i,j} a_i\otimes a_j$ belongs to $T_{R_Q}(a_1,\ldots,a_r)$. The matrix $((\ep_{i,j}))$ is symmetric if $n$ is even and skew-symmetric if $n$ is odd. As $|a_i|=n-1$, the element $l$ lies in free graded Lie algebra over $R_Q$ generated by $a_1,\ldots,a_r$ equalling 
$$l= \sum_{i<j} \ep_{i,j} [a_i,a_j]^{gr}.$$
The superscript $gr$ is used to denote the graded bracket. Consider the graded Lie algebra $\LL_r^{gr}$ (over $R_Q$) given by 
$$\LL^{gr}_r= \Lie^{gr}_{R_Q}(a_1,\ldots,a_r)/(l).$$
We make an analogous ungraded construction. Consider the element 
$$l^u= \sum_{i<j} \ep_{i,j} [a_i,a_j]\in \Lie_{R_Q}(a_1,\ldots,a_r). $$
 Note that $l^u$ equals $l$ if $n$ is odd. Analogously denote $\LL_r^u= \frac{\Lie(a_1,\ldots,a_r)}{(l^u)}$.

The Lie algebra  $\LL_r^u$ has a grading with $|a_i|=n-1$. Denote by $\LL_{r,w}^u$ the degree $w$ homogeneous elements of $\LL_r^{u}$. From Proposition \ref{freemod} and Theorem \ref{Liebasis} we know that $\LL^u_r$ is a free $R_Q$-module and the Lyndon basis gives a basis of $\LL^u_r$.

List the elements of the Lyndon basis in order as $l_1 < l_2 <\ldots$ and define $h_i= h(l_i)=r+1$ if $b(l_i) \in \LL_r^u$, so that $h(l_i)\leq h(l_{i+1})$. Note that $b(l_i)$ represents an element of $\Lie(a_1,\ldots,a_k)$ and is thus represented by an iterated Lie bracket of $a_i$. Use the iterated Whitehead products to define maps $\lambda_i : S^{h_i}\rightarrow X$. 



\begin{theorem}\label{cellhtpy}
There is an isomorphism
$$\pi_*(X)\otimes R_Q \cong \sum_{i\geq 1} \pi_*( S^{h_i})\otimes R_Q$$ 
and the inclusion of each summand on the right hand side\footnote{Note that the right hand side is a finite direct sum for each $\pi_n(X)\otimes R_Q$.} is given by $\lambda_i$. 
\end{theorem}

\begin{proof}
Write $S(t) = \prod_{i=1}^t \Omega S^{h_i}$. As in Theorem \ref{htpy}, we obtain a map 
$$\Lambda: S := \mathit{hocolim}~ S(t) \rightarrow \Omega X$$   
We prove that $\Lambda_*$ is an isomorphism in $H_*(-;R_Q)$-homology. There are two cases: $n$ is even and $n$ is odd.

We know from Theorem \ref{cellloop} that $H_*(\Omega X; R_Q) \cong T_{R_Q}(a_1,\ldots,a_r)/(l)$. When $n$ is odd $l=l^u$ which lies in the Lie algebra generated by $a_1,\ldots,a_r$.  Thus 
$$H_*(\Omega X; R_Q) \cong T_{R_Q}(a_1,\ldots,a_r)/(l)$$
is the universal enveloping algebra of the Lie algebra $\LL_r^u \cong \Lie(a_1,\ldots a_r)/ (l^u)$. By the Poincar\'e-Birkhoff-Witt Theorem,  $E_0H_*(\Omega X;R_Q)$ is the symmetric algebra on $\LL_k^u$. 

We also have
$$H_*(S;R_Q) \cong T_{R_Q} (c_{h_1 -1})\otimes_{R_Q} T_{R_Q}(c_{h_2-1})\ldots \cong R_Q[c_{h_1-1},c_{h_2 -1},\ldots ].$$
Each $c_{h_i-1}$ maps to the  Hurewicz image of $\rho (\lambda_i)\in H_{h_i-1}(\Omega X)$. The map $\rho$ carries each $\alpha_i$  to $a_i$. The element  $b(l_i)$ is mapped inside $H_*(\Omega X)$ to the element corresponding to the graded Lie algebra element upto a sign. 

Denote the Lie algebra element (ungraded) corresponding to $b(l_i)$ by the same notation. It follows from \eqref{rhobl} that
$$ E_0T_{R_Q}(a_1,\ldots,a_k)/(l) \cong R_Q [b(l_1),b(l_2),\ldots] \cong  R_Q [\rho(b(l_1)),\rho(b(l_2)),\ldots].$$
Hence monomials in $\rho(b(l_i))$ are a basis of $T_{R_Q}(a_1,\ldots,a_r)/(l)$. The map $\Lambda : S\to \Omega X$ maps $c_{h_i-1} \to \rho(b_i)$ and takes the product of elements in $R_Q[c_{h_1-1},c_{h_2 -1},\ldots ]$ to the corresponding Pontrjagin product. It follows that $\Lambda_*$ is a $H_*(-;R_Q)$-isomorphism. Thus, $\Lambda$ induces a weak equivalence between the localizations with respect to $H(-;R_Q)$. As $\Lambda$ is a map between simple spaces, $\Lambda_*$ is an isomorphism on $\pi_*(-)\otimes R_Q$. 

Now let $n$ be even. As in the odd case, it suffices to show that $T(a_1,\ldots,a_r)/(l)$ has a basis given by monomials on $\rho(b(l_1)),\rho(b(l_2)),\ldots$ where $\rho(b(l_i))$ is mapped as the above. First observe that $ T(a_1,\ldots,a_r)/(l)$ is universal enveloping algebra of the graded Lie algebra $\LL^{gr}_r$ so that the Poincar\'e-Birkhoff-Witt Theorem for graded Lie algebras implies 
$$E(\LL^{gr}_r)^{\mathit{odd}}  \otimes P(\LL^{gr}_r)^{\mathit{even}} \cong  E_0 T(a_1,\ldots,a_r)/(l).$$
As in Theorem \ref{htpy}, we conclude $\rho(b(l_i))$ span $T(a_1,\ldots,a_r)/(l)$. 


Hence we have that the graded map
$$ R_Q[\rho(b(l_1)),\rho(b(l_2)),\ldots] \to   E_0  T_{R_Q}(a_1,\ldots,a_r)/(l)$$
is surjective. We also know 
$$  R_Q[b(l_1),b(l_2),\ldots]    \to E_0 T_{R_Q}(a_1,\ldots,a_r)/(l^u)$$
is an isomorphism. Now both $T_{R_Q}(a_1,\ldots,a_r)/(l)$ and $T_{R_Q}(a_1,\ldots,a_r)/(l^u)$ have bases given by the Diamond lemma\footnote{We apply Proposition \ref{freemod}. Note that the case $r=2$ needs to be dealt with differently. In that case we use $\Z_{(p)}$ for $p \notin \Pi_Q$ in place of $R_Q$ for which the statement about Diamond lemma holds. Assimilating these results over the different primes $p\notin \Pi_Q$ we may complete the proof of this part of the Theorem for $r=2$.} and thus are of the same graded dimension. It follows that the graded pieces of $ R_Q[\rho(b(l_1)),\rho(b(l_2)),\ldots]$ and $T_{R_Q}(a_1,\ldots,a_r)/(l)$ have the same rank. Thus on graded pieces one has a surjective map between free $R_Q$-modules of the same rank. Such a map must be an isomorphism.
\end{proof}

The same proof for $(n-1)$-connected $2n$-manifolds implies the following result.
\begin{theorem}\label{cellhtpyform}
 The number of groups $\pi_s( S^m)\otimes R_Q$ in $\pi_s(X)\otimes R_Q$ is $0$ if $m$ is not of the form $d(n-1)+1$ and if $m=d(n-1)+1$ this number is 
$$\sum_{c|d}\frac{ \mu(c)}{c}\sum_{a+2b=\frac{d}{c}}(-1)^b{a+b \choose b}\frac{r^a}{a+b}.$$

\end{theorem}

\begin{exam}
The theorem implies that if the rank of the form $Q\otimes \Q$ is at least $2$ then after inverting a finite set of primes the homotopy groups depend only on the number $r$ of $n$-spheres. We note that inverting these primes are necessary. Consider the complexes $X_f= S^2\vee S^2 \cup_f e^4$ where $f\in \pi_3 (S^2\vee S^2)= \Z\{\eta_1\}\oplus \Z\{\eta_2\}\oplus \Z\{[\iota_1,\iota_2]\}$ ($\eta_i$ are the Hopf invariant one maps on either factor and $[\iota_1,\iota_2]$ is the Whitehead product of the identity maps on either factor). Take $f_1 = p[\iota_1,\iota_2]$ and $f_2=p^2[\iota_1,\iota_2]$. For $X_{f_i}$ ($i=1,2$) rank of $Q\otimes \Q$ equals $2$ and $\Pi_Q=\{p\}$. Therefore $\pi_*X_{f_1} \otimes \Z[\frac{1}{p}] \cong   \pi_*X_{f_2} \otimes \Z[\frac{1}{p}]$ by Theorem \ref{cellhtpyform}. However we have $\pi_3X_{f_1} \cong \Z^2 \oplus \Z/p\Z$ while $\pi_3X_{f_2} \cong \Z^2 \oplus \Z/p^2\Z$ so that their $p$-local homotopy groups are not isomorphic. 
\end{exam}
\begin{exam}
In a similar vein, consider $X_f= S^2\vee S^2 \cup_f e^4$ where $f\in \pi_3 (S^2\vee S^2)$ as before with 
$$f_1=p^2[\iota_1,\iota_1]+[\iota_2,\iota_2],\,\,f_2=p[\iota_1,\iota_2].$$
The absolute value of the determinant of both the (non-singular) intersection matrices are $p^2$. However, $\pi_3X_{f_1} \cong \Z^2$ while $\pi_3X_{f_2} \cong \Z^2 \oplus \Z/p\Z$ so that their $p$-local homotopy groups are not isomorphic. This shows that inverting some prime is necessary even when the intersection form is non-singular (over $\Q$).
\end{exam}

\begin{theorem}\label{MooreCW}
Let $X$ be a CW complex with $r\geq 3$ $n$-cells and one $2n$-cell. Then $X$ is rationally hyperbolic and has unbounded $p$-primary torsion for all but finitely many primes. 
\end{theorem}
\begin{proof}
It is clear from the formula in Theorem \ref{cellhtpyform} that $X$ is rationally hyperbolic if $r\geq 3$. Now there are three cases depending on the $\rank(Q\otimes\Q)$ being $0$, $1$ and at least $2$. In the last case it follows from methods similar to Theorem \ref{Moorehcm} that $\pi_\ast S^l \otimes R_Q$ occurs as a summand of $\pi_\ast M\otimes R_Q$ for arbitrarily large $l$. Therefore, the $p$-exponents are unbounded for any primes not in $\Pi_Q$. 

Now we deal with the cases rank $0$ and $1$ by proving that for all but finitely many primes $S^n_{(p)} \vee S^n_{(p)}$ is a retract of $X_{(p)}$. The result then follows from the formulas in \cite{Hil55}. If $\rank(Q\otimes\Q)$ is $0$ or $1$, then one may choose a basis $v_1,\cdots, v_r$ of $H^n(X;\Q)$ so that $v_1^2=v_2^2=v_1v_2=0$ (this is possible because $r\geq 3$). Let $A$ be the change of basis matrix. Choose a prime $p$ so that 
\begin{enumerate}
\item $v_1,\cdots, v_r$ is a basis for $H^n(X;\Z_{(p)})$.
\item There is no $p$-torsion in $\pi_{2n-1}S^n$. 
\end{enumerate} 
Note that these conditions are satisfied for all but finitely many $p$. We prove the above statement for all these primes. 

Let $X= (\vee_r S^n) \cup_f e^{2n}$. Now we map $Y= (\vee_r S^n)_{(p)} \cup_g e^{2n} \to X_{(p)}$ by mapping the wedge of $n$-spheres along $(v_1,\cdots, v_r)$ and putting $g= f \circ A^{-1}$. Thus, $Y$ is well-defined in the category of $p$-local spaces (or rational spaces) and in these cases the map $Y\to X$ is a $p$-local equivalence (respectively, rational equivalence).

Refer to the notation $l$ from Proposition \ref{cellcobar}. It follows that in terms of the basis $v_1,\cdots,v_r$ the coefficients of $v_1^2,v_1\otimes v_2, v_2\otimes v_1, v_2^2$ in $l$ is $0$. Let $\alpha_1, \cdots \alpha_r$ be the homotopy classes $\alpha_i : S^n \stackrel{v_i}{\to} Y$. Now we have that the element 
$$g \in \oplus_{i=1}^r (\alpha_i \circ \pi_{2n-1} S^n_{(p)})  \oplus (\oplus_{i<j} \Z_{(p)}\{[\alpha_i,\alpha_j]\})$$
and similarly its rationalization
$$g_\Q \in  \oplus_{i=1}^r (\alpha_i \circ \pi_{2n-1} S^n_{\Q}) \oplus (\oplus_{i<j} \Q\{[\alpha_i,\alpha_j]\}).$$

The rational class $g_\Q$ is $0$ in the homotopy Lie algebra of $X$. By the Milnor-Moore theorem, the symmetric algebra on the homotopy Lie algebra is isomorphic to the homology $H_*(\Omega X;\Q)$. Reading off relations in $H_{2n-2}(\Omega X;\Q)$ we note that $g_\Q$ must be $0$ in this group while from the proof of Proposition \ref{cellcobar} we note that the only relations are given by multiples of $l$. It follows that the factors in $g_\Q$ along the summands $\alpha_1 \circ \pi_{2n-1} S^n_{\Q}$,  $\alpha_2 \circ \pi_{2n-1} S^n_{\Q}$ and $\Q\{[\alpha_1,\alpha_2]\}$ are $0$. 

From the second condition on $p$ above, $\pi_{2n-1} S^n_{(p)}$ is free. Since $g_\Q = g \otimes \Q$ and $\pi_{2n-1} S^n_{(p)}\to \pi_{2n-1} S^n_{\Q}$ is injective, we deduce that the factors of $g$ along the summands  $\alpha_1 \circ \pi_{2n-1} S^n_{(p)}$,  $\alpha_2 \circ \pi_{2n-1} S^n_{(p)}$ and $\Z_{(p)}\{[\alpha_1,\alpha_2]\}$ are $0$. Therefore $g$ may be expressed as 
$$g= \sum_{i=3}^r \alpha_j\circ \lambda_j + \sum_{i<j, (i,j)\neq (1,2)} \mu_{i,j} [\alpha_i,\alpha_j]$$ 





Now we may construct a map $h: Y\to S^n_{(p)}\vee S^n_{(p)}$ as follows. On the $n$-skeleton consider the map $h: \vee_r S^n_{(p)} \to S^n_{(p)} \vee S^n_{(p)}$ which projects off the last $r-2$ factors, i.e., $\alpha_1$ and $\alpha_2$ go to the two inclusions and $\alpha_j \mapsto 0$ for $j\ge 3$.  Now from the formula above the composite $h\circ g :S^{2n-1} \to S^n_{(p)}\vee S^n_{(p)}$ is null-homotopic. Thus we obtain $h : Y \to S^n_{(p)}\vee S^n_{(p)}$ which gives a retraction upto homotopy.  
\end{proof}

\begin{rmk}\label{MooreX}
When $r=1$ the space $X=S^n\cup_f e^{2n}$. Note that $\pi_{2n-1}(S^n)$ is finite when $n$ is odd, and has a $\Z$-summand when $n$ is even. If $f$ is of finite order\footnote{By \cite{Hil55} and the growth of $p$-exponents as dimension $l$ of $S^l$ increases it follows that $S^n \vee S^m$ has unbounded $p$-exponents for any prime as long as $n,m\geq 2$. The case when $f=0$ corresponds to $X=S^n\vee S^{2n}$.} then on the one hand the rational homotopy groups of $X$ are the same as that of the rationally hyperbolic space $X=S^n\vee S^{2n}$ and on the other hand, by \cite{NeSe81}, $X$ has unbounded $p$-exponents for odd primes. When $f$ is of infinite order, and necessarily $n$ is even, then 
$$\pi_\ast(X)\otimes\Q\cong \pi_\ast(S^n\cup_h  e^{2n})\otimes \Q$$
where $h:S^{2n-1}\to S^n$ generates the $\Z$-summand of $\pi_{2n-1}(S^n)$. It is clear (via minimal models) that $S^n\cup_h e^{2n}$ is rationally elliptic. When $h$ is a suitable multiple of $[\iota_{n},\iota_{n}]$ then \cite{NeSe81} proves that $X$ has finite $p$-exponents for odd primes. Some very special cases of $3$-cell complexes are also covered in \cite{NeSe81} but to our knowledge the general case for $r=2$ is unresolved.
\end{rmk}

\subsection{Decomposition of loop space}\label{OmM}
This section relates the results of \cite{BeTh14} to the methods of the current paper. The paper \cite{BeTh14} shows that in various cases homotopy groups of different manifolds are isomorphic. This is done by proving that the loop spaces are weakly equivalent. In the examples of this paper, we have shown that the loop spaces are products of loop spaces of spheres. Therefore these can be used to deduce the loop space decompositions in \cite{BeTh14}. This also covers the $8,16$ dimensional cases which was not proved in \cite{BeTh14}. 

We know that the cases when Betti number is $1$ and Betti number is at least $2$ are very different. In the latter case we have the following theorem.

\begin{theorem}\label{htpybetti}
 If $M$ and $N$ are $(n - 1)$-connected $2n$-dimensional manifolds such that $H_n(M) \cong H_n(N)\cong\Z^r$ with rank $r\ge 2$ then $\Omega M \simeq  \Omega N$. As a consequence, the homotopy groups of an $(n-1)$-connected $2n$-manifold depend only on the $n^\textit{th}$ Betti number $r$ if it is at least $2$. If $M$ and $N$ are $n$-dimensional connected sums of sphere products, then $\Omega M \simeq \Omega N$ if and only if $H_m(M ) \cong H_m (N)$ for each $m < n$.
\end{theorem}

\begin{proof}
For  an $(n-1)$-connected $2n$-manifold $M$ with $n^\textup{th}$ Betti number $\ge 2$, we have the formula of $\Omega M$ from Theorem \ref{htpy} as a product of spheres. Clearly this formula only depends on the $n^\textup{th}$ Betti number of $M$. This proves the first two statements of the theorem. For the statement on products of spheres refer to Theorem \ref{htpyT}.
\end{proof}

Next we deduce the result on the decomposition of the loop space of \cite{BeTh14}.

\begin{theorem}
Let $M_r$ be an $(n-1)$-connected $2n$-dimensional manifold such that $\dim (H_n(M)) = r\ge 2$. There is a homotopy equivalence
$$\Omega M \simeq \Omega (S^n \times S^n ) \times \Omega (J \vee (J \wedge \Omega(S^n \times S^n ))) $$
where $J = \vee_{k=2}^r S^n$. 
\end{theorem}

\begin{proof}
Note that the homotopy type of $\Omega M$ depends only on $r$ by the proof of Theorem \ref{htpy}. Also observe that the formula in Theorem \ref{cellhtpy} can be made integral if $\Pi_Q$ is empty and this matches with the formula in Theorem \ref{htpy}. The set of primes is empty if there is a $2\times 2$ minor for the cup product form whose determinant is $\pm 1$. The condition is satisfied by the space $S^n\times S^n \vee J \simeq \vee_r S^n \cup_{[\alpha_1,\alpha_2]} e^{2n}$.  Hence
$$\Omega M_r \simeq \Omega ((S^n \times S^n) \vee_{k=2}^r S^n) \simeq \Omega ((S^n \times S^n) \vee J) .$$
Now we apply Theorem 2.6 of \cite{BeTh14} with $P=(S^n\times S^n) \vee J$ to obtain the result.
\end{proof}

The Betti number $1$ case is quite different and for each $n=2,4,8$ the situation is slightly different. The $n=2$ case is already dealt with in \cite{BeTh14}. For $n=4,8$ there are different manifolds with Betti number $1$ whose homotopy groups are not isomorphic. We have the following theorem.

\begin{theorem}
 Let $M$ be a $(n - 1)$-connected $2n$-manifold with $H_n(M) \cong \Z$.\\
\textup{(a)} If $n=2$, then $M\simeq \C P^2$. Hence $\Omega M \simeq S^1\times \Omega S^5$ and $\pi_k M \cong \pi_k S^5$ if $k\geq 3$ and $\pi_2 M \cong \Z$. \\
\textup{(b)} If $n=4$, $M\simeq V_{m,1}^8$ for some $m\in \Z/12\Z$. If $m \equiv 0~\mbox{or}~2~(\mbox{mod}~3)$, $\Omega V_{m,1}^8 \simeq S^3 \times \Omega S^{11}$ and 
$$\pi_kV_{m,1}^8 \cong \pi_{k-1}S^3 \oplus \pi_k S^{11}.$$
If $m,m'$ are congruent to $1~(\mbox{mod}~3)$ then $\pi_*V_{m,1}^8 \cong \pi_*V_{m',1}^8$ and these are not equal to  $\pi_{k-1}S^3 \oplus \pi_k S^{11}$. There is an equivalence after inverting $3$, i.e., 
$$\Omega V_{m,1}^8 \simeq_{\Z[1/3]} S^3 \times \Omega S^{11}.$$
\textup{(c)} If $n=8$, $M\simeq V_{m,1}^{16}$ for some $m\in \Z/120\Z$. In this case we have an equivalence after inverting $6$, $\Omega V_{m,1}^{16} \simeq_{\Z[1/6]} S^7 \times \Omega S^{23}$. Thus, 
$$\pi_kV_{m,1}^{16} \otimes \Z[1/6] \cong (\pi_{k-1}S^7 \oplus \pi_k S^{23})\otimes \Z[1/6].$$
 
\end{theorem}

\begin{proof}
The formula in (a) is well known. The results in (b) follows from Theorem \ref{4assoc} and Theorem \ref{4bet1htpy}. The proof of these Theorems actually prove the mentioned loop space decompositions. 
The result in (c) follows from Theorem \ref{8bet1htpy}. 
\end{proof}

Similar results can be obtained for CW complexes of the type $X_r \simeq \vee_r S^n \cup e^{2n}$. 
\begin{theorem}
For a CW complex $X_r$ as above, suppose that the cup product bilinear form $Q$ has rank at least $2$. There exists a finite set of primes $\Pi_Q$ so that after localization with respect to $H(-;R_Q)$, 
$$\Omega X_r \simeq_{R_Q} \Omega (J \vee (J \wedge \Omega(S^n \times S^n ))) $$
where $R_Q$ is the subring of $\Q$ obtained by inverting the primes in $\Pi_Q$. 
\end{theorem}
\begin{proof}
The result directly follows from the proof of Theorem \ref{cellhtpy}.
\end{proof}


\bibliographystyle{siam}

\vspace*{1.5cm}

\noindent{\small D}{\scriptsize EPARTMENT OF }{\small M}{\scriptsize ATHEMATICS, }{\small R}{\scriptsize AMAKRISHNA }{\small M}{\scriptsize ISSION }{\small V}{\scriptsize IVEKANANDA }{\small U}{\scriptsize NIVERSITY, }{\small H}{\scriptsize OWRAH, }{\small WB} {\footnotesize 711202, }{\small INDIA}\\
{\it E-mail address} : \texttt{samik.basu2@gmail.com}, \,\,\texttt{samik@rkmvu.ac.in}\\[0.2cm]

\noindent{\small D}{\scriptsize EPARTMENT OF }{\small M}{\scriptsize ATHEMATICS, }{\small R}{\scriptsize AMAKRISHNA }{\small M}{\scriptsize ISSION }{\small V}{\scriptsize IVEKANANDA }{\small U}{\scriptsize NIVERSITY, }{\small H}{\scriptsize OWRAH, }{\small WB} {\footnotesize 711202, }{\small INDIA}\\
{\it E-mail address} : \texttt{basu.somnath@gmail.com}, \,\,\texttt{somnath@rkmvu.ac.in}\\[0.2cm]

\end{document}